\documentclass[a4paper,11pt]{article}

\usepackage{amsfonts,amsmath,amssymb,amsthm}
\usepackage{thmtools}                              % cref fix: keep after amsthm
\usepackage{mathtools}
\usepackage{esint}
\usepackage{mathrsfs}
\usepackage[mathscr]{eucal}                    % \mathscr comes from here

\usepackage[margin=2.8cm]{geometry}
\usepackage{enumitem}
\usepackage{wrapfig}
\usepackage[font=small,justification=centering]{caption}
\usepackage{xspace}

\usepackage[dvipsnames]{xcolor}
\usepackage{tikz} \usetikzlibrary{angles,quotes,calc,intersections}

\usepackage[backend=biber, url=false, style=numeric-comp, sorting=nyt, giveninits=true,isbn=false]{biblatex}
\DeclareFieldFormat{prefixnumber}{\mkbibbold{#1}}
\DeclareFieldFormat{labelnumber}{\mkbibbold{#1}}

\usepackage[hidelinks]{hyperref}
\usepackage{cleveref}
\newtheorem{theorem}{Theorem}[section]
\newtheorem{proposition}[theorem]{Proposition}

\newtheorem{lemma}[theorem]{Lemma}
\newtheorem{corollary}[theorem]{Corollary}

\newtheorem{problem}[theorem]{Problem}
\theoremstyle{definition}
\newtheorem{definition}[theorem]{Definition}
\newtheorem{assumption}[theorem]{Assumption}\newtheorem*{assumption*}{Assumption}
\newtheorem{example}[theorem]{Example}

\theoremstyle{remark}
\newtheorem{remark}[theorem]{Remark}
\theoremstyle{plain}
\newtheorem{thmx}{Theorem}

\crefformat{equation}{(#2#1#3)}
\Crefformat{equation}{(#2#1#3)}
\crefformat{theorem}{Theorem~#2#1#3}
\Crefformat{theorem}{Theorem~#2#1#3}
\crefformat{thmx}{Theorem~#2#1#3}
\Crefformat{thmx}{Theorem~#2#1#3}
\crefformat{lemma}{Lemma~#2#1#3}
\Crefformat{lemma}{Lemma~#2#1#3}
\crefformat{proposition}{Proposition~#2#1#3}
\Crefformat{proposition}{Proposition~#2#1#3}
\crefformat{remark}{Remark~#2#1#3}
\Crefformat{remark}{Remark~#2#1#3}
\crefformat{definition}{Definition~#2#1#3}
\Crefformat{definition}{Definition~#2#1#3}
\crefformat{example}{Example~#2#1#3}
\Crefformat{example}{Example~#2#1#3}
\crefformat{problem}{Problem~#2#1#3}
\Crefformat{problem}{Problem~#2#1#3}
\crefformat{assumption}{Assumption~#2#1#3}
\Crefformat{assumption}{Assumption~#2#1#3}
\crefformat{section}{Section~#2#1#3}
\Crefformat{section}{Section~#2#1#3}
\crefformat{corollary}{Corollary~#2#1#3}
\Crefformat{corollary}{Corollary~#2#1#3}

\DeclareMathOperator{\rank}{rank}
\DeclareMathOperator{\dist}{dist}
\DeclareMathOperator{\proj}{proj}
\DeclareMathOperator{\myspan}{span}
\DeclareMathOperator{\id}{id}
\DeclareMathOperator{\prob}{Prob}
\DeclareMathOperator{\Var}{Var}

\newcommand{\Asp}{\mathcal{A}}
\newcommand{\Afv}{{\mathbb{A}}}
\newcommand{\Bsp}{\mathcal{B}}
\newcommand{\Mfv}{\mathbb{M}}

\newcommand{\sou}{\mathbb{V}}
\newcommand{\tar}{\mathbb{W}}
\newcommand{\presou}{\mathbb{U}}
\newcommand\R{\mathbb R}

\newcommand{\var}{X}
\newcommand{\varalt}{Y}
\newcommand{\pol}{\var_0}
\newcommand{\dircone}{Z}
\newcommand{\compSys}{\mathcal{K}}
\newcommand{\lin}[1]{\ell_{#1}}
\newcommand{\slin}[1]{\ell^+_{#1}}
\newcommand{\bdsubset}{K}

\newcommand{\Meas}{\mathcal{M}}
\newcommand{\ball}{B}
\newcommand{\sphere}{S}
\newcommand{\resmes}{\mathbin{\vrule height 1.6ex depth 0pt width 0.18ex\vrule height 0.15ex depth 0pt width 1.1ex}}
\newcommand{\predual}{\mathbf{E}}
\newcommand{\predualAqc}{\predual_{\Asp}}
\newcommand{\Ebf}{\mathbf E}
\newcommand{\Env}{\mathscr{C}}
\newcommand{\QC}{\mathcal{Q}}
\newcommand{\regC}{C}
\newcommand{\ckorn}{\varpi}
\newcommand{\const}[1]{\ifx\\#1\\\omega\else \omega(#1)\fi}

\newcommand{\toweakstar}{\stackrel{*}\rightharpoonup}
\newcommand{\st}{\, : \,}
\newcommand{\mysm}{\smallsetminus}
\newcommand{\dpr}[1]{\left<#1\right>}
\newcommand\eps{\varepsilon}
\newcommand{\Raita}{Rai\c{t}\u{a}\xspace}

\makeatletter
\renewcommand*\author[1]{%
	\stepcounter{author}%
	\ifnum\c@author=1
	\gdef\@author{#1}%
	\else
	\xdef\@author{\unexpanded\expandafter{\@author\and#1}}%
	\fi
	\csgdef{author@\the\c@author}{#1}}
\newcommand*\email[1]{%
	\csgdef{email@\the\c@author}{#1}}
\newcommand*\address[1]{%
	\csgdef{address@\the\c@author}{#1}}
\AtEndDocument{%
	\xdef\author@count{\the\c@author}%
	\c@author=1
	\print@authors}
\newcommand*\print@authors{%
	\ifnum\c@author>\author@count
	\else
	\print@author{\the\c@author}%
	\advance\c@author by 1
	\expandafter\print@authors
	\fi}
\newcommand*\print@author[1]{%
	\par\medskip
	\begin{tabular}{@{}l@{}}%
		\textsc{\csuse{author@#1}}\\
		\csuse{address@#1}\\
		E-mail address:
		\href{mailto:\csuse{email@#1}}{\csuse{email@#1}}
\end{tabular}}
\makeatother

\title{Quantified rigidity of {$\mathcal{A}$}-free Young measure concentrations}
\author{Adolfo Arroyo-Rabasa}
\address{Dipartimento di Matematica\\
	Universit\`a di Pisa\\
	Largo Bruno Pontecorvo 5\\
	56127 Pisa, Italy}
\email{adolfo.rabasa@unipi.it}
\author{Averil Aussedat}
\address{Dipartimento di Matematica\\
	Universit\`a di Pisa\\
	Largo Bruno Pontecorvo 5\\
	56127 Pisa, Italy}
\email{averil.aussedat@dm.unipi.it}
\begin{document}
\maketitle
\begin{abstract}
We establish quantitative stability of the De Philippis--Rindler rigidity on conical convex regions. By constructing $\mathcal{A}$-quasiconvex functions with a tailored concavity---which additionally yields a new, elementary proof of the original theorem for constant-rank operators---we prove a spreading inequality that strictly limits the angular concentration of $\mathcal{A}$-free measures. This geometric bound establishes that singular concentrations cannot cluster arbitrarily close to a direction outside the Tartar wave cone. As a consequence, we obtain a new $L^1$ compensated compactness result: $\mathcal{A}$-free sequences with uniformly bounded mass are forced to be equi-integrable, thereby preventing the formation of mass concentrations, provided their targets are asymptotically restricted to cones whose aperture is controlled by a power of the distance to the wave cone of $\mathcal A$.

\end{abstract}
%%%%%%%%%%%%%%%%%%%%%%%%%%%%%%%%%%%%%%%%%%%%%%%%%%%%%%%%%%%%%%%%%%%%%%%%%%%%%%
%%  1.  INTRODUCTION
%%%%%%%%%%%%%%%%%%%%%%%%%%%%%%%%%%%%%%%%%%%%%%%%%%%%%%%%%%%%%%%%%%%%%%%%%%%%%%
\section{Introduction}
The rank-one theorem stands as a cornerstone in the fine structure theory of functions of bounded variation, with far-reaching consequences for the calculus of variations and geometric measure theory. Established by Alberti~\cite{albertiRankOneProperty1993} (see also~\cite{deLellis_note_rank_one,massaccesi_vittone}), the result reveals a striking algebraic rigidity governing the singular behavior of curl-free measures (we refer the reader to \Cref{sec:prel} for basic notations):
\begin{theorem}[Rank-one theorem, Alberti \cite{albertiRankOneProperty1993}]\label{thm:rank_one}
	Let $\Omega \subset \R^d$ be an open set and let $u: \Omega \to \R^m$ be a locally integrable field whose distributional gradient $Du$ is represented by an $\R^{m \times d}$-valued locally finite Radon measure. Then,
	\begin{equation}\label{eq:geometric}
		\operatorname{rank} \left( \frac{Du}{|Du|}(x) \right) = 1 \quad \text{for } |Du|^s\text{-almost every } x \in \Omega.
	\end{equation}
\end{theorem}
Over the last decade, it has become increasingly clear that this measure-theoretic result is not an isolated phenomenon, but rather the blueprint for a much broader rigidity shared by a large class of PDE constraints. To understand how this rigidity extends beyond curl-free measures, one may view the phenomenon through the lens of $L^1$-approximate \emph{compensated compactness}. Given two finite dimensional normed real vector spaces $\sou,\tar$ of dimensions $N$ and $M$, which we can identify with $\R^N$ and $\R^M$, the fundamental question is to identify the sets $\compSys \subset \sou$ for which a sequence $(v_j)_j \subset L^1(\Omega,\sou)$ satisfying the compensated system
\begin{equation}\label{eq:compensated_system}
	\begin{cases}
		\begin{aligned}
			\Asp v_j &= 0 \quad \text{in the sense of distributions,}\\
			v_j &\toweakstar v \quad \text{in the sense of measures,}\\
			\quad \dist(v_j, \compSys) &\to 0 \quad \text{in $L^1(\Omega)$,}
		\end{aligned}
	\end{cases}
\end{equation}
automatically upgrades from weak-$*$ convergence to either weak $v_j\rightharpoonup v$ or strong $v_j \to v$ convergence in $L^1$, thereby preventing the formation of mass concentrations.
In this general framework, $\Asp$ denotes a system (of $M$ equations over $N$ unknowns) of the form
\begin{equation}\label{eq:A}
	\Asp v = \sum_{|\alpha|= k} A_\alpha \partial^\alpha v, \qquad A_\alpha \in \mathrm{Hom}(\sou,\tar),
\end{equation}
acting on vector-valued distributions $v \in \mathcal D'(\Omega,\mathbb V)$, and $\compSys \subset \mathbb V$ denotes a set of preferred target configurations. By transitioning from the classical gradient setting ($\R^{m \times d}$) to this broader context ($\sou$), the curl-free condition is replaced by the general $\Asp$-free constraint. Consequently, the cone of rank-one matrices is naturally superseded by the Tartar wave cone $\Lambda_{\Asp} \subset \sou$ (see~\cite{tartarCompensatedCompactnessApplications1979}). This cone consists of those amplitudes $\pol \in \sou$ that support one-dimensional \emph{$\Asp$-free laminates}---that is, constant states for which there exists a spatial direction $\xi_0 \in \sphere^{d-1}$ such that for any bounded measurable profile $h : \R \to \{0,1\}$,
\[
\Asp(\pol \, h(x \cdot \xi_0)) = 0 \quad \text{in the sense of distributions.}
\]
In a far-reaching generalization, De Philippis and Rindler~\cite{dephilippisStructureFreeMeasures2016} (see also \cite{arroyorabasaDimensionalEstimatesRectifiability2019}) showed that the wave cone is the essential locus for the singular polar density of any $\Asp$-free measure:
\begin{theorem}[De Philippis--Rindler \cite{dephilippisStructureFreeMeasures2016}]\label{thm:DPR}
	Let $\Asp$ be an operator as in~\eqref{eq:A} and let $\mu \in \Meas(\Omega,\sou)$ be such that $\Asp \mu = 0$ in the sense of distributions. Then,
	\begin{equation}
		\frac{\mu}{|\mu|}(x) \in \Lambda_{\Asp} \qquad \text{for } |\mu|^s\text{-almost every } x \in \Omega.
	\end{equation}
\end{theorem}
The decisive step in De Philippis and Rindler's approach was to bypass the classical geometric tools available only for gradients (such as the area and co-area formulas) in favor of pseudo-differential calculus. They recognized that the rank-one theorem is fundamentally a consequence of a compensated compactness property: $\Asp$-free fields cannot concentrate mass when asymptotically close to a ray generated by an amplitude strictly outside the wave cone. Specifically, their proof reveals that if $\dircone \in \sou \smallsetminus \Lambda_{\Asp}$ and $(v_j)_j$ satisfies~\eqref{eq:compensated_system} with the half-line $\compSys = \slin{\dircone}$, where
\[
\slin{\dircone} \coloneqq \{ t \, \dircone : t \ge 0 \},
\]
then $v_j \to v$ strongly in $L^1$.
Despite its reach, \Cref{thm:DPR} remains a standalone achievement that resolves a compensated compactness system in the critical $L^1$ setting only for the rigid geometry of a one-dimensional half-line. A fundamental open question is whether the De Philippis--Rindler rigidity is structurally stable: can we prevent mass concentrations along $\Asp$-free fields whose targets are asymptotically restricted to \emph{convex cones} of positive aperture and disjoint from the wave cone?
Expanding the target geometry from a ray to a cone poses a substantial analytic challenge. In the critical $L^1$ setting, the strong regularizing effects of constant-coefficient linear elliptic systems break down, mimicking an elliptic system with $L^1$ coefficients where one can generally only establish convergence in measure \cite[Lemma~2.7]{mullerVariationalModelsMicrostructure1999}. Because convergence in measure cannot prevent mass concentrations, exact rigidity is sharply lost: the equi-integrability of a weakly converging sequence $v_j \rightharpoonup v$ in $L^1$ may fail entirely \cite{kristensenCharacterizationGeneralizedGradient2010, arroyo-rabasaCharacterizationGeneralizedYoung2021, arroyo-rabasaHigherIntegrabilityMeasures2024}. Furthermore, this geometric generality dismantles the strict one-dimensional sign constraints underpinning the De Philippis--Rindler proof, exposing the problem to pathologically unbounded singular integrals.
The main goal of this work is to establish a quantifiable stability of the De Philippis--Rindler rigidity on conical convex regions, by providing a strict lower bound for the allowed aperture in terms of the cone's distance to the wave cone $\Lambda_\Asp$.
%%  1.1  Setup and main results  ---------------------------------------------
\subsection{Setup and main results}
To bypass these analytic barriers, we study the fine-scale limits of $L^1$-mass concentrations generated by $\Asp$-free sequences using the framework of generalized Young measures. To capture the escaping of mass to infinity, we radially compactify the state space $\sou$ by attaching the unit sphere of directions at infinity, $\sphere_\infty \cong \sphere_{\sou}$, and consider the angular concentration probability measures $\nu^\infty \in \prob(\sphere_{\sou})$ associated with $\Asp$-free generalized Young measures on $\Omega$ (see \Cref{sec:YM} for precise definitions).
Throughout this work, we operate under the standing assumption that the differential operator $\Asp$ satisfies the \emph{constant-rank property} \cite{SW_constant_rank,muratCompaciteCompensationCondition1981}:
\begin{assumption}\label{ass:constant_rank}
	The symbol matrix $\Afv(\xi) \coloneqq \sum_{|\alpha| = k} A_{\alpha} \xi^{\alpha}$ has constant rank over $\R^d \mysm \{0\}$.
\end{assumption}
Under this structural hypothesis, a key feature of the angular concentration measures---derived in \Cref{sec:Aqcf}---is that their collective state space is convex and compact (with respect to a suitable topology). This topological and geometric regularity, which may be interesting in its own right, is one of the simplifying elements enabling our core arguments:
\begin{thmx}[Convexity of $\Asp$-free concentrations]\label{prop:convexity}
	Let $\Asp$ satisfy the constant rank property. The set of all angular concentration probability measures $\nu^\infty_x \in \prob(\sphere_{\sou})$ that arise, at $\lambda$-almost every point $x$, from an $\Asp$-free generalized Young measure $(\nu,\lambda,\nu^\infty) \in \mathbf Y_\Asp(\Omega)$ is convex and weak-$*$ compact.
\end{thmx}

This set is described explicitly in \Cref{def:micro}, where it is denoted $\Env_\Asp$; the identification and its proof are contained in \Cref{sec:YM}. A related convexity statement for the class of concentrations generated by sequences converging to zero in measure is contained in~\cite[Proposition~2.2]{gennaioliConcentrationPhenomenaVanishing2026}; we compare the two in \Cref{sec:sota}.
\begin{remark}[Global convexity]
	This unconditional convexity is somewhat surprising. In the classical theory, the convexity of $\Asp$-free (or gradient) Young measures typically holds only among measures that share a common macroscopic barycenter measure $\dpr{\id,\nu_x} \, dx + \dpr{\id,\nu^\infty_x} \lambda(dx)$. The fact that angular concentration profiles form a convex set---independent of their underlying macroscopic averages---highlights a rigid structural decoupling that occurs purely at the concentration level.
\end{remark}
The cornerstone of our approach is to study this convex set of concentration measures by duality with Morrey's concept of \emph{quasiconvexity}~\cite{Morrey_52}, generalized to the $\Asp$-free setting by Fonseca and M\"uller~\cite{fonsecaAQuasiconvexityLowerSemicontinuity1999}. This duality can be expressed through a well-known Jensen inequality connecting the macroscopic barycenter $\bar \nu^\infty$ to the concentration profile~\cite{arroyo-rabasaCharacterizationGeneralizedYoung2021,kristensenOscillationConcentrationSequences2022}: any such concentration measure $\nu^\infty$ necessarily satisfies (see \Cref{res:KR} and \Cref{sec:YM})
\begin{equation}\label{eq:nec}
	f(\bar \nu^{\infty}) \le f(0) + \dpr{f^\infty,\nu^\infty}
\end{equation}
for all $\Asp$-quasiconvex functions $f\in \Ebf(\sou)$.
The true power of~\eqref{eq:nec} emerges when paired with the geometric rigidity of positively $1$-homogeneous $\Asp$-quasiconvex functions (cf.~\cite{Sverak_quasinconvex_sublinear}). A key insight due to Kirchheim and Kristensen~\cite{KK1,kirchheimRankOneConvex2016} establishes that such functions are necessarily convex on points of the Tartar wave cone $\Lambda_{\Asp}$. We observe that this property actually \emph{characterizes} the wave cone (see \Cref{res:spreading}\ref{item:spreading:exif}): $\varalt \in \Lambda_{\Asp}$ if and only if
\begin{equation}\label{eq:sub}
	f(\var + \varalt) \le f(\var) + f^\infty(\varalt) \quad \text{for all } \var \in \sou,
\end{equation}
for all $\Asp$-quasiconvex integrands $f \in \Ebf(\sou)$. The sufficiency is \cite{kirchheimRankOneConvex2016} and gives the inequality for \emph{every} $\var$; the necessity is \Cref{res:spreading}\ref{item:spreading:exif} and only needs to violate it at $\var=0$.
\begin{remark}[Proof of the De Philippis--Rindler theorem]
	This perspective yields a direct and elementary proof of \Cref{thm:rank_one} and \Cref{thm:DPR} for constant-rank operators. Mimicking the $(L^1,L^{1,\infty})$-estimate constructions of M\"uller~\cite{mullerQuasiconvexFunctionsWhich1992}, one can tailor an $\Asp$-quasiconvex integrand that develops a strict concavity outside the wave cone. This geometric penalty rules out the possibility of a pure singular concentration on any ray $\ell_{\pol}$ with $\pol \notin \Lambda_\Asp$. Crucially, this requires only the classical definition of $\Asp$-quasiconvexity evaluated on periodic fields, bypassing the machinery of generalized Young measures entirely. We present this self-contained proof in \Cref{sec:dpr}.
\end{remark}
%%  * The spreading inequality  ----------------------------------------------
\subsubsection*{The spreading inequality}
Extending this rigidity to full conical neighborhoods requires a quantitative leap, which we achieve by refining the techniques of M\"uller~\cite{mullerQuasiconvexFunctionsWhich1992} and Zhang~\cite{zhang_constructions}.
For an elliptic subspace $V \subset \sou$---meaning $V \cap \Lambda_{\Asp} = \{0\}$---we measure the distance from the subspace to the wave cone on the unit sphere of $\sou$ as follows:
\[
d(V, \Lambda_{\Asp}) \coloneqq \inf_{\var \in V \cap \sphere_{\sou}} \, \dist(\var, \Lambda_{\Asp}) \in [0,1].
\]
We then define a Korn-type constant $\const{V}$ as
\begin{equation}
	\const{V} \coloneqq \sup \left\{ \frac{\|v\|_{L^{1,\infty}}}{\|\dist(v,V)\,\|_{L^1}} : v \in \regC^{\infty}_c(\R^d,\sou) \cap \ker \Asp,\ v \not\equiv 0 \right\}.
\end{equation}
The size of $\const{V}$ measures how badly the regularizing effect of $\Asp$ degenerates near $V$: it is the price at which an $\Asp$-free field is allowed to live close to $V$ in $L^1$ while still carrying weak-$L^1$ mass, so that a \emph{small} $\const{V}$ forbids $\Asp$-free functions from concentrating arbitrarily close to $V$. Crucially, in \Cref{res:korntype} we show that $\const{V}$ blows up with order at most $d(V,\Lambda_{\Asp})^{-(\lfloor d/2 \rfloor + 2)}$ as $V$ approaches the wave cone. Thus, by testing with compact perturbations of $\dist(\,\cdot\,,V)$, we are able to quantify a strict structural lower bound on the dispersion of mass at infinity:

\begin{thmx}[Spreading inequality]\label{thm:spreading}
	Let $\nu^\infty \in \prob(\sphere_{\sou})$ be the angular concentration measure of a homogeneous $\Asp$-free generalized Young measure $(\nu,\lambda,\nu^\infty) \in \mathbf Y_\Asp(\Omega)$ with $\lambda(\Omega) > 0$. If $V \subset \sou$ is a subspace with
	\[
	V \cap \Lambda_{\Asp} = \{0\},
	\]
	then
	\begin{equation}\label{intro:spreading:statement}
		|P_V \pol| - (1+\const{V}) |P_{V^\perp} \pol| \le \const{V} \int_{\sphere_{\sou}} \dist(\var, V) \, d\nu^\infty(\var)
	\end{equation}
	where $\pol = \bar\nu^\infty$ is the barycenter of $\nu^\infty$. If $\pol \in V$, then the estimate above reduces to
	\[
	|\pol| \le \const{V} \int_{\sphere_{\sou}} \dist(\var, V) \, d\nu^\infty(\var).
	\]
\end{thmx}
This spreading inequality imposes a stronger constraint than the De Philippis--Rindler theorem. It dictates that the mass of a concentration measure $\nu^\infty$ cannot cluster arbitrarily close to its own projection onto an elliptic plane. As a consequence, since a probability measure on the unit sphere has variance $\Var(\nu^\infty) = \int_{\sphere_{\sou}} |\var - \pol|^2 \, d\nu^\infty(\var) = 1 - |\pol|^2$, we obtain the barycenter bound
\begin{align*}
	|\pol|^2 \leqslant \frac{1}{1 + (\const{\lin{\pol}})^{-2}} \ <\ 1.
\end{align*}
Combining it with the obtained Korn-type constant (\Cref{res:korntype}), we discover that an angular concentration whose barycenter avoids the wave cone must spread on the sphere at a rate controlled by the distance to $\Lambda_\Asp$ (see \Cref{rem:boundPol} and \Cref{rem:variance}):
\begin{align}\label{eq:variance}
	\Var(\nu^\infty) \ \geqslant\ \frac{d(\lin{\pol},\Lambda_\Asp)^{2(\lfloor d/2\rfloor + 2)}}{1 + \ckorn^2} .
\end{align}

\begin{remark}[Rigidity of 1-dimensional concentrations]
	As an immediate consequence of \Cref{thm:spreading}, if the concentration profile $\nu^\infty$ of an $\Asp$-free Young measure is supported entirely on a line $V = \lin{\pol}$ meeting the wave cone only at the origin, the distance integral vanishes, forcing its barycenter to be identically zero. Consequently, the \emph{only} such concentration is the sum of equally weighted Diracs at the intersections of $V$ with the unit sphere: $\nu^\infty = \frac{1}{2}(\delta_{{\pol}/{|\pol|}} + \delta_{-{\pol}/{|\pol|}})$.
\end{remark}
%%  * Quantifying the aperture of cones  -------------------------------------
\subsubsection*{Quantifying the aperture of cones}
With the spreading inequality in hand, we can now evaluate the rigidity of specific conical geometries. Let $V \subset \sou$ be an \emph{elliptic subspace}. Given aperture angles $0 \le \alpha, \beta < \pi/2$, we consider the class $\mathcal C_{\alpha,\beta}(V)$ of all \emph{closed, convex} and \emph{acute} cones $\mathcal K \subset \sou$ satisfying the \emph{ellipticity condition}
\begin{equation}\label{eq:cone_ell}
	\mathcal K \cap \Lambda_{\Asp} \subset \{0\}
\end{equation}
and the aperture constraints:
\begin{align}\label{def:cone_constraints}
	\left\{\begin{aligned}
		|P_V X| &\ge \cos(\alpha)|X| \\
		\dpr{X,Z} &\ge \cos(\beta) |X|
	\end{aligned}
	\right. \quad \qquad \text{for all $X \in \mathcal K$, \quad for some $Z \in V \cap S_{\sou}$}.
\end{align}
~\vspace{1.5em}
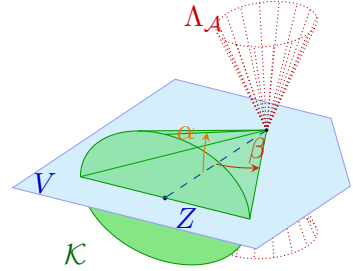
\begin{wrapfigure}{r}{0.45\textwidth}
	\centering
	\vspace{-3.6em}
	\begin{tikzpicture}[
		x={(-0.6cm,-0.4cm)}, y={(0.8cm,-0.2cm)}, z={(0cm,1cm)},
		scale=0.56,
		>=stealth
		]
		\def\Ry{2.5}
		\def\Rz{1.5}
		\def\L{4}
		\def\Wz{2.5} \def\Wr{1.2}
		\def\angContour{108.9}
		\draw[red!70!black, densely dotted]
		plot[domain=0:360, samples=40] ({\Wr*cos(\x)}, {\Wr*sin(\x)}, -\Wz);
		\foreach \ang in {0,22.5,...,337.5} {
			\draw[red!70!black, densely dotted] (0,0,0) -- ({\Wr*cos(\ang)}, {\Wr*sin(\ang)}, -\Wz);
		}
		\fill[white] (0,0,0) -- plot[domain=180:352.6, samples=40] (\L, {\Ry*cos(\x)}, {\Rz*sin(\x)}) -- cycle;
		\fill[white] plot[domain=180:360, samples=40] (\L, {\Ry*cos(\x)}, {\Rz*sin(\x)}) -- cycle;
		\fill[green!80!black, fill opacity=0.3]
		plot[domain=180:360, samples=40] (\L, {\Ry*cos(\x)}, {\Rz*sin(\x)}) -- cycle;
		\fill[green!80!black, opacity=0.25]
		(0,0,0) -- plot[domain=180:352.6, samples=40] (\L, {\Ry*cos(\x)}, {\Rz*sin(\x)}) -- cycle;
		\draw[green!60!black] plot[domain=180:360, samples=40] (\L, {\Ry*cos(\x)}, {\Rz*sin(\x)}) -- cycle;
		\filldraw[cyan!15, draw=blue!40, opacity=1]
		(-1.2, -3.6, 0) -- (5.2, -3.6, 0) -- (5.2, 3.6, 0) -- (1.5, 3.6, 0) -- (-1.2, 1.0, 0) -- cycle;
		\node[blue!80!black] at (5, -2.8, 0.2) {$V$};
		\draw[dashed, blue!80!black] (0,0,0) -- (\L, 0, 0) node[anchor= north west]{$\dircone$};
		\fill[blue!80!black] (\L, 0, 0) circle (1.6pt);
		\draw[green!60!black] (0,0,0) -- (\L, \Ry, 0);
		\draw[green!60!black] (0,0,0) -- (\L, -\Ry, 0);
		\draw[->, red] plot[domain=0:32, samples=20] ({2*cos(\x)}, {2*sin(\x)}, 0);
		\node[red, anchor=south west] at (2.2, 0.8, 0) {$\beta$};
		\draw[red!70!black, densely dotted]
		plot[domain=0:360, samples=40] ({\Wr*cos(\x)}, {\Wr*sin(\x)}, \Wz);
		\foreach \ang in {0,22.5,...,337.5} {
			\draw[red!70!black, densely dotted] (0,0,0) -- ({\Wr*cos(\ang)}, {\Wr*sin(\ang)}, \Wz);
		}
		\node[red!80!black, anchor=south west] at (1, -2, \Wz-0.4) {$\Lambda_{\Asp}$};
		\fill[green!80!black, opacity=0.3]
		plot[domain=0:180, samples=40] (\L, {\Ry*cos(\x)}, {\Rz*sin(\x)}) -- cycle;
		\fill[green!80!black, opacity=0.25]
		(0,0,0) -- plot[domain=0:108.9, samples=40] (\L, {\Ry*cos(\x)}, {\Rz*sin(\x)}) -- cycle;
		\draw[green!60!black] plot[domain=0:180, samples=40] (\L, {\Ry*cos(\x)}, {\Rz*sin(\x)}) -- cycle;
		\draw[green!60!black] (0,0,0) -- (\L, 0, \Rz);
		\draw[green!60!black] (0,0,0) -- (\L, {\Ry*cos(\angContour)}, {\Rz*sin(\angContour)});
		\draw[->, orange!90!black] plot[domain=0:20.5, samples=20] ({2.5*cos(\x)}, 0, {2.5*sin(\x)});
		\node[orange!90!black, anchor=south east] at (2.4, 0, 0.5) {$\alpha$};
		\node[green!40!black, fill=white, fill opacity=0.7, text opacity=1, inner sep=1pt, rounded corners]
		at (3, -3.4, -2.4) {$\mathcal K$};
		\fill[black] (0,0,0) circle (1.2pt);
	\end{tikzpicture}
	\caption{An admissible cone $\mathcal K$ satisfying~\eqref{eq:cone_ell}--\eqref{def:cone_constraints}.}\label{fig:cone}
\end{wrapfigure}
Here $P_V$ denotes the orthogonal projection onto $V$. Geometrically, \eqref{def:cone_constraints} describes a cone (see \Cref{fig:cone}) formed by bounding the out-of-plane deviation from $V$ by the angle $\alpha$, intersected with the $\beta$-cone opening radially around the ray $\ell_Z^+$. Restricting $\beta$ to be strictly less than $\pi/2$ already ensures the cone is acute. If $V = \ell_Z$, setting $\alpha = \beta$ recovers the standard radially symmetric cone of aperture $\alpha$ around $\ell_Z^+$.
We are now in place to state our main compensated compactness result, which follows directly from evaluating the spreading inequality against these geometric constraints.
\begin{thmx}[Compensated compactness on quantified cones]\label{thm:main}
	Let $\Asp$ be an operator as in~\eqref{eq:A} satisfying the constant rank property. There exists a positive constant $\ckorn = \ckorn(\Asp,d)$ with the following property: if $V \subset \sou$ is a subspace such that
	\[
	V \cap \Lambda_{\Asp} = \{0\},
	\]
	and $0 \le \alpha,\beta < \pi/2$ are angles satisfying 	\begin{align}\label{main:hyp}
		0 \leqslant \alpha < \frac{\cos(\beta)}{1+2\ckorn} \cdot d(V, \Lambda_{\Asp})^{\lfloor d/2 \rfloor + 2},
	\end{align}
	then for any $\mathcal K \in \mathcal C_{\alpha,\beta}(V)$ and any uniformly bounded sequence $(\mu_j)_j \subset \Meas(\Omega;\sou)$ satisfying
	\begin{equation*}
		\begin{cases}
			\begin{aligned}
				\Asp \mu_j &= 0 \quad \text{in the sense of distributions,}\\
				\mu_j &\toweakstar \mu \quad \text{in the sense of measures,}\\
				\quad \int_{\Omega} \dist\Big(\tfrac{d\mu_j}{d|\mu_j|},\mathcal K\Big) \, d|\mu_j| &\longrightarrow 0 ,
			\end{aligned}
		\end{cases}
	\end{equation*}
there exists a map $v \in L^1(\Omega,\sou)$ such that $\mu = v \, \mathscr L^d$ and
	\begin{align*}
		|\mu_j|^s \to 0 \; \; \text{strongly in $\Meas(\Omega)$} \qquad \text{and} \qquad
		\frac{d\mu_j}{d\mathscr L^d} \rightharpoonup v \; \; \text{weakly in } L^1_{{\mathrm{loc}}}(\Omega;\sou).
	\end{align*}
\end{thmx}

\begin{remark}[Operators of first order]\label{rem:GR}
	When $\Asp$ has order one, the aperture condition~\eqref{main:hyp} can be dispensed with altogether. After the first version of this work appeared, Gennaioli and Rindler~\cite[Theorem~1.8]{gennaioliConcentrationPhenomenaVanishing2026} proved that for a first-order constant-rank operator the conclusion of \Cref{thm:main} holds for \emph{every} closed convex cone $\mathcal K$ with $\mathcal K \cap \Lambda_{\Asp} = \{0\}$ and $\mathcal K \cap (-\mathcal K) = \{0\}$, with no restriction whatsoever on the aperture. Their argument proceeds by a smoothing along the heat flow which, as they observe, is tied to the first order; for operators of order $k \ge 2$ the condition~\eqref{main:hyp} remains, to our knowledge, the only available criterion. We discuss this in \Cref{sec:sota}.
\end{remark}
\begin{remark}[Symmetric cones and one-dimensional subspaces]
	When the target cone is aperture-symmetric, meaning $\alpha = \beta$, condition~\eqref{main:hyp} becomes implicit in $\alpha$; it may then be replaced by the explicit sufficient condition
	\[
	0 \le \alpha < \frac{\pi}{4(1+2\ckorn)} \, d(V, \Lambda_{\Asp})^{\lfloor d/2 \rfloor + 2},
	\]
	proved in \Cref{rem:equal}.
	This is particularly relevant when the elliptic subspace $V$ is one-dimensional (i.e.\ $V = \lin{\dircone}$): in that regime, setting $\alpha = \beta$ means that the pure distance aperture bound applies to the radially symmetric cone $\{\var \in \sou : \dpr{\var,\dircone} \ge \cos(\alpha)|\var| \}$.
\end{remark}

\begin{remark}
If the weak-$*$ convergence $\mu_j \toweakstar \mu$ is not assumed, the remaining hypotheses still give that the sequence of maps $(d\mu_j/d\mathscr L^d)_j$ is equi-integrable on the compact subsets of $\Omega$, while the singular parts converge to zero in total variation; this is \Cref{res:main:measure} below, together with Step~1 in the proof of \Cref{thm:main}, which is global and uses neither localization nor the weak-$*$ convergence. Here equi-integrability is understood as in \cite[Definition~1.26]{ambrosioFunctionsBoundedVariation2000}.
\end{remark}

\Cref{thm:main} is optimal in the following two respects; both statements are about sequences bounded in $L^1$, which is the regime the theorem addresses.
\begin{itemize}[leftmargin=*,itemsep=0.1cm]
	\item \textit{No strong $L^1$-compactness, and no higher integrability.} Because, in general, a cone $\mathcal K \in \mathcal C_{\alpha,\beta}(V)$ has a nonempty interior, strong $L^1$-compactness fails due to the persistence of purely oscillatory bounded laminates: if $\ball_r(\dircone) \subset \mathcal K$ and $\zeta \in \Lambda_\Asp \mysm \{0\}$, then $\dircone \pm \frac r2 \frac{\zeta}{|\zeta|}$ both lie in $\mathcal K$ and differ by an element of the wave cone, so that $\mathcal K$ always supports a non-trivial $\Asp$-free laminate. Moreover, no higher integrability may be added to the conclusion: see \Cref{rem:no_higher_int}.

	\item \textit{Acuteness cannot be dropped.} The aperture condition is sharp in $\beta$: if $\beta = \pi/2$ then $\mathcal K$ may contain a line through the origin, and there exist gradient sequences that fail to be locally equi-integrable on any open subset (see also \cite{arroyo-rabasaCharacterizationGeneralizedYoung2021,kristensenCharacterizationGeneralizedGradient2010} and the explicit example in \cite[Section~7]{arroyo-rabasaHigherIntegrabilityMeasures2024}). This is consistent with \Cref{thm:spreading}: for a symmetric concentration such as $\frac12(\delta_{\pol} + \delta_{-\pol})$ the barycenter vanishes and the spreading inequality is empty. The angle $\alpha$, by contrast, enters only through the threshold~\eqref{main:hyp}: acuteness is already secured by $\beta < \pi/2$, and \Cref{rem:GR} shows that for first-order operators no bound on $\alpha$ is needed at all.

\end{itemize}

\begin{example}[Gradients]\label{ex:curl}
	Let $\Asp = \operatorname{curl}$ act on $\R^{m \times d}$, so that $\Lambda_{\Asp} = \{a \otimes b : a \in \R^m,\ b \in \R^d\}$ is the cone of rank-one matrices, and equip $\R^{m\times d}$ with the Frobenius inner product. The distance of a matrix $\var$ to the rank-one cone is $(\sum_{i \ge 2} \sigma_i(\var)^2)^{1/2}$, where $\sigma_1(\var) \ge \sigma_2(\var) \ge \cdots$ are the singular values of $\var$. Hence, for a line $V = \lin{\var}$ with $|\var| = 1$,
	\[
	d(\lin{\var},\Lambda_{\Asp}) = \sqrt{1 - \sigma_1(\var)^2},
	\]
	which is positive precisely when $\rank \var \ge 2$. More generally, for a subspace $V$,
	\[
	d(V,\Lambda_{\Asp}) = \Big(1 - \sup\{\sigma_1(\var)^2 : \var \in V, \ |\var| = 1\}\Big)^{1/2}.
	\]
	Take $d = m = 2$ and let $V = \myspan\{I,J\}$ be the plane of conformal matrices, where $I$ is the identity and $J = \left(\begin{smallmatrix} 0 & -1\\ 1& 0\end{smallmatrix}\right)$. Every $aI + bJ$ is a multiple of a rotation, so both its singular values equal $\sqrt{a^2+b^2}$; in particular every nonzero element has rank two, $V$ is elliptic, and $\sigma_1(\var) = 1/\sqrt 2$ whenever $|\var| = 1$. Therefore
	\[
	d(V,\Lambda_{\Asp}) = \tfrac{1}{\sqrt 2},
	\]
	and since $\lfloor d/2 \rfloor + 2 = 3$, \Cref{thm:main} applies to every $\mathcal K \in \mathcal C_{\alpha,\beta}(V)$ with $\alpha < \cos(\beta)\,2^{-3/2}/(1+2\ckorn)$; in the symmetric case $\alpha = \beta$ this reads $\alpha < \pi \, 2^{-3/2}/(4(1+2\ckorn))$. Thus curl-free measures whose polar densities asymptotically concentrate in a sufficiently thin cone about the conformal plane cannot develop singular parts.
\end{example}

%%  1.2  State of the art  ----------------------------------------------------
\subsection{State of the art}\label{sec:sota}
A purely qualitative version of \Cref{thm:main} follows from the De Philippis--Rindler theorem by a contradiction argument along targets collapsing onto a ray outside the wave cone; being non-constructive, it leaves the threshold aperture unquantifiable. Supplying the explicit power-law dependence on $d(V,\Lambda_\Asp)$ is the role of \Cref{thm:spreading}.

As recorded in \Cref{rem:GR}, no threshold at all is needed in the first-order case: Gennaioli and Rindler~\cite[Theorem~1.8]{gennaioliConcentrationPhenomenaVanishing2026} obtain the conclusion of \Cref{thm:main} for every first-order constant-rank operator and every closed convex cone containing no line. Their route is entirely different from ours---a Choquet-type decomposition of concentration Young measures, with a heat-flow smoothing in place of the classical Fourier constructions---and yields much besides, notably a resolution of Bouchitt\'e's vanishing mass conjecture~\cite{bouchitteOptimizationLightStructures2003,bouchitteOptimizationLightStructures2020} in that setting.

We also record that~\cite[Proposition~2.2]{gennaioliConcentrationPhenomenaVanishing2026} contains a convexity statement in the spirit of \Cref{prop:convexity}. The two are logically independent: \Cref{prop:convexity} covers the concentration profiles of \emph{all} $\Asp$-free generalized Young measures, for constant-rank operators of every order, and asserts weak-$*$ compactness as well---which the authors observe they are not able to reach---whereas theirs addresses those generated by sequences converging to zero \emph{in measure}, but dispenses with the constant-rank hypothesis in the first-order case. Bridging the two classes is precisely where the constant rank property is spent on our side, through the decoupling of oscillation from concentration of \Cref{res:KR}.

In 2018, Filip Rindler (personal communication~\cite{Rindler2018}) posed the following conjectural generalization, part of a broader inquiry into the fine structure of PDE-constrained measures:
\begin{problem}[Quantitative rigidity, Rindler \cite{Rindler2018}]\label{prob:rindler}
	Does the distance of a polar density $\frac{d\mu}{d|\mu|}$ to the wave cone $\Lambda_{\Asp}$ control---perhaps in a nonlinear way---how ``close'' $\mu$ is to being singular?
\end{problem}
He also proposed a weaker and more concrete formulation, for a bounded sequence $(v_j)_j \subset L^1(\Omega;\sou)$ of smooth fields with $\Asp v_j = 0$:
\begin{equation}\label{eq:rindler_weak}
	\text{if } \ \dist(v_j,\Lambda_\Asp) \ \ge\ \delta \, |v_j| \ \text{ a.e., for some } \delta > 0, \ \text{ must } (v_j)_j \text{ be equi-integrable?}
\end{equation}
The target it singles out, $\compSys_\delta \coloneqq \{ \var \in \sou \st \dist(\var,\Lambda_\Asp) \ge \delta|\var| \}$, meets $\Lambda_\Asp$ only at the origin but is \emph{symmetric}, hence neither convex nor acute: in the language of \eqref{def:cone_constraints} it corresponds to $\beta = \pi/2$, the borderline excluded by our second sharpness comment, and it equally violates the one-sidedness required in~\cite[Theorem~1.8]{gennaioliConcentrationPhenomenaVanishing2026}. Thus \eqref{eq:rindler_weak} appears to remain open. Both results answer it affirmatively along any single convex acute branch of $\compSys_\delta$; neither excludes that two antipodal branches cooperate to build a concentration whose barycenter vanishes, and against that possibility \Cref{thm:spreading} provides some quantitative information, albeit implicitly.

\Cref{rem:no_higher_int} bears on the formulation of \Cref{prob:rindler} itself. In the $L^1$-approximate regime, ``how close $\mu$ is to being singular'' cannot be measured on the Lebesgue scale: the crests constructed there stay within an arbitrarily small $L^1$-distance of a fixed elliptic direction while blowing up in every $L^p$ with $p>1$, and even in $L\log L$. Any quantitative form of \Cref{prob:rindler} at $p=1$ must therefore be phrased in terms of equi-integrability, or of the size of the singular part, rather than of integrability gains---in contrast with the range $p>1$, where such gains are exactly what one obtains~\cite{bateAlbertiRepresentationsRectifiability2025,arroyo-rabasaHigherIntegrabilityMeasures2024}.

It is useful, finally, to distinguish the \emph{exact} and the \emph{approximate} constraints. Under an exact ($L^\infty$) constraint into a convex target $\compSys$ with $\compSys \cap \Lambda_\Asp \subset \{0\}$, the absence of rank-one connections heavily restricts oscillation, a phenomenon studied at length for material microstructures~\cite{ballFinePhaseMixtures1987,tartarRemarksSeparatelyConvex1993,bhattacharyaRestrictionsMicrostructure1994,zhangMultiwellProblemsRestrictions1994,mullerVariationalModelsMicrostructure1999}. Approximate ($L^p$, $p>1$) constraints let the target be approached only asymptotically, and have produced quantitative absolute continuity and higher integrability results for divergence-free symmetric positive tensors~\cite{serreDivergencefreePositiveSymmetric2018,bateQuantitativeAbsoluteContinuity2020}, determinants~\cite{mullerSurprisingHigherIntegrability1989, Coifmanetal1993}, and other anisotropic systems~\cite{arroyo-rabasaHigherIntegrabilityMeasures2024, guerraCompensationPhenomenaConcentration2024, dephilippisMichaelSimonInequalityAnisotropic2026}; Bate and Weigt~\cite{bateAlbertiRepresentationsRectifiability2025} cover the full range of $L^p$ while keeping the strict target $\compSys = \slin{\dircone}$, quantifying how far an $\Asp$-free function is from $L^p$ by its $L^1$ distance to the ray. The regime addressed here, $p=1$, is the critical threshold.

%%  Structure of the paper  --------------------------------------------------
\subsection*{Structure of the paper}
\Cref{sec:prel} fixes notation and collects the ingredients we need: the compensated Korn inequality with its quantitative degeneracy (\Cref{res:korntype}) and its periodic counterpart (\Cref{res:transfer}); the functional-analytic setting of integrands with linear growth and $\Asp$-quasiconvexity, together with the Kristensen--\Raita decoupling (\Cref{res:KR}) and the convexity statement (\Cref{res:convexity}); and the theory of $\Asp$-free generalized Young measures (\Cref{sec:YM}), where \Cref{prop:convexity} is proved. \Cref{sec:constructions} contains the heart of the paper: the construction of $\Asp$-quasiconvex integrands with a tailored concave perturbation, and the resulting spreading inequality (\Cref{res:spreading}). \Cref{sec:dpr} uses that construction to give a short, Young-measure-free proof of \Cref{thm:DPR} for constant-rank operators. \Cref{sec:proofs} proves \Cref{thm:spreading} and \Cref{thm:main}. Two appendices collect the derivative bounds for Moore--Penrose pseudo-inverses and the transfer of the Korn inequality to the torus.
%%%%%%%%%%%%%%%%%%%%%%%%%%%%%%%%%%%%%%%%%%%%%%%%%%%%%%%%%%%%%%%%%%%%%%%%%%%%%%
%%  2.  PRELIMINARIES
%%%%%%%%%%%%%%%%%%%%%%%%%%%%%%%%%%%%%%%%%%%%%%%%%%%%%%%%%%%%%%%%%%%%%%%%%%%%%%
\section{Preliminaries}\label{sec:prel}
In all that follows, $\Omega \subset \R^d$ will denote an open Lipschitz domain. The open ball of center $a$ and radius $r > 0$ is denoted $\ball_r(a)$ or simply $\ball_r$ when $a = 0$. If $E$ is a Banach space, we denote by $\left<\cdot,\cdot\right>$ the duality product with $E^*$. If moreover $E$ is an inner product space, we canonically identify $E^*$ with $E$ so that $\dpr{\cdot,\cdot}$ also denotes its inner product. For brevity, the unit ball of $E$ is shortened as $\ball_{E}$ and its unit sphere as $\sphere_{E}$.
The set of nonnegative measures over $\Omega$ is denoted $\Meas_+(\Omega)$ and we also write $\Meas(\Omega,E)$ to denote the set of $E$-valued measures. Given a non-trivial vector-valued measure $\mu$ on $\Omega$, we write $|\mu|$ to denote its associated positive measure and $\mu/|\mu|$ to denote the Radon--Nikodym derivative of $\mu$ with respect to $|\mu|$. Lastly, we write $\mu = \mu^a + \mu^s$ to denote the Radon--Nikodym--Lebesgue decomposition of $\mu$ into its absolutely continuous and singular parts.
Henceforth, $\sou, \tar$ are finite dimensional inner product spaces.
We shall work with a constant coefficient, linear and homogeneous partial differential operator
\begin{align}\label{def:A}
	\Asp = \sum_{|\alpha| = k} A_\alpha \partial^\alpha, \qquad A_\alpha \in \mathrm{Hom}(\sou,\tar),
\end{align}
where as usual $\alpha = (\alpha_1,\dots,\alpha_d) \in \mathbb{N}^d$ are $d$-tuples with modulus $|\alpha| = \alpha_1 + \cdots + \alpha_d$ and $\partial^\alpha = \partial_1^{\alpha_1} \circ \cdots \circ \partial_d^{\alpha_d}$ denotes the composition of distributional partial derivatives with respect to the index $\alpha$. The operator $\Asp$ acts on $\sou$-valued tempered distributions over $\R^d$ as a $k$-homogeneous Fourier multiplier:
\[
(\Asp u)^{\widehat{}}(\xi) = (2\pi \mathrm i)^k \Afv(\xi) \widehat u(\xi),
\]
where the principal symbol map $\Afv : \R^d \mysm \{0\} \to \mathrm{Hom}(\sou,\tar)$, defined as
\[
\Afv(\xi) = \sum_{|\alpha|=k} A_\alpha \xi^\alpha \qquad \text{for any } \xi^\alpha = \xi_1^{\alpha_1} \cdots \xi_d^{\alpha_d},
\]
is the so-called (principal) symbol associated with $\Asp$. The Tartar wave cone \cite{tartarCompensatedCompactnessApplications1979}
\[
\Lambda_{\Asp} \coloneqq \bigcup_{|\xi|=1} \ker \Afv(\xi) \subset \sou
\]
is the algebraic set containing all vectors where $\Asp$ fails to be regularizing over planewaves with fixed vectorial part, i.e., $\pol \in \Lambda_{\Asp}$ if and only if there exists $\xi \in \R^d \mysm \{0\}$ such that $\Asp (\pol \, \mathrm{e}^{2\mathrm i\pi \left<\boldsymbol{\cdot},\xi\right>}) = 0$.
%%  2.1  Compensated Korn inequalities  --------------------------------------
\subsection{Compensated Korn inequalities}
We start by obtaining Korn-type inequalities for solutions of $\Asp u = 0$. The underlying idea is the following: if $V \cap \Lambda_{\Asp} = \{0\}$, then the operator $\Asp$ can be complemented with the projection on the orthogonal of $V$ to get an elliptic operator. The classical Calderón--Zygmund estimates then allow one to control the $L^{1,\infty}$-norm of an $\Asp$-free function by that of its projection on $V^{\perp}$, as in \eqref{eq:korntype:statement} below. Here, as usual
\begin{align*}
	\|u\|_{L^{1,\infty}} \coloneqq \sup_{\lambda > 0} \lambda \, \mathscr L^d(\{|u| > \lambda\}).
\end{align*}
This argument appears at various places in the case of the row-wise curl operator, although in perhaps implicit formulations (\cite[Th.~3]{tartarCompensatedCompactnessMethod1983}, \cite[Th.~4.1]{bhattacharyaRestrictionsMicrostructure1994}, \cite{dipernaCompensatedCompactnessGeneral1985}). Here we provide a proof for constant-rank operators valid for higher-dimensional elliptic planes. The novelty of our construction is the use of an explicit multiplier construction, opposed to the classical na\"ive pseudo-inverse approach, which allows us to obtain the crucial distance estimate on the degeneracy of the estimate as $V$ approaches the wave cone:
\begin{lemma}[Compensated Korn inequality]\label{res:korntype}
	Let $\Asp$ be an operator satisfying the constant-rank \Cref{ass:constant_rank} and let $V \subset \sou$ be a subspace such that
	\[
	V \cap \Lambda_{\Asp} = \{0\}.
	\]
	Then, the Korn constant
	\begin{align*}
		\const{V} \coloneqq \sup \left\{ \frac{\|u\|_{L^{1,\infty}}}{\|\dist(u, V)\|_{L^1}} \st u \in \regC_c^{\infty}(\mathbb{R}^d;\sou) \cap \ker \Asp,\ u \not\equiv 0 \right\}
	\end{align*}
	is finite and satisfies
	\begin{align}\label{eq:korntype:statement}
		\const{V} \le \ckorn \, d(V,\Lambda_{\Asp})^{-(\lfloor \frac{d}{2} \rfloor + 2)}
	\end{align}
	for some constant $\ckorn \ge 1$ depending solely on the coefficients of $\Afv$ and the dimension $d$.
\end{lemma}
\begin{remark}
	Establishing $\const{V} < \infty$ for elliptic subspaces is relatively simple (by mimicking M\"uller's trick via weak $L^1$ estimates). In particular, the proofs of \Cref{res:spreading}\ref{item:spreading:exif} and the De Philippis--Rindler Theorem for constant rank operators follow from relatively simple arguments. Establishing upper bounds for $\const{V}$ in terms of $d(V,\Lambda_\Asp)$ is what requires a considerably more delicate analysis.
\end{remark}
\begin{proof}
	Denote by $P_V$ and $P_{V^\perp}$ the orthogonal projections on $V$ and its orthogonal complement, and $\Pi(\xi) = P_{\ker \Afv(\xi)^\perp}$, the projection on $\ker \Afv(\xi)^\perp$, as a subspace of $\sou$. The idea is to construct a parameterized multiplier $\Mfv(\xi) : \sou \to \sou$ solving the algebraic equation
	\begin{align}\label{korntype:defM}
		\Mfv(\xi) P_{V^\perp}\widehat u(\xi) = \widehat u(\xi) \qquad \text{for all $u$ satisfying $\widehat u(\xi) \in \ker \Afv(\xi)$,}
	\end{align}
	and show that $\Mfv(\cdot)$ satisfies the assumption of Mikhlin's theorem.
	Let $u$ be an $\Asp$-free Schwartz map, $v \coloneqq P_{V^\perp} u$, and $w \coloneqq P_V u$, so that $\widehat u = \widehat v + \widehat w$. Because $\widehat u(\xi) \in \ker \Afv(\xi)$ for all $\xi \in \R^d \mysm \{0\}$, there holds $\Pi(\xi) \widehat u = 0$, or equivalently $\Pi(\xi) \widehat w = - \Pi(\xi) \widehat v$. Taking the projection on $V$, we get
	\begin{align}\label{korntype:projV}
		P_V \Pi(\xi) P_V \widehat w = - P_V \Pi(\xi) \widehat v.
	\end{align}
	Denote $W(\xi) \coloneqq \Pi(\xi) P_{V} : \sou \to \sou$. Then $W(\xi)^* W(\xi) = P_V \Pi(\xi) P_V$ is invertible on $V$, since the kernel of $W(\xi)^* W(\xi)|_V$ is $\ker \Afv(\xi)$, which intersects $V$ only at $0$. In consequence, we may solve \eqref{korntype:projV} with the Moore-Penrose pseudo-inverse $(W(\xi)^* W(\xi))^{-1} W(\xi)^*$, to obtain
	\begin{align*}
		\widehat w(\xi) = - W(\xi)^{\dagger} \widehat v(\xi) \quad \implies \quad \widehat u(\xi) = \widehat v(\xi) - W(\xi)^{\dagger} \widehat v(\xi).
	\end{align*}
	Therefore the multiplier in \eqref{korntype:defM} can be taken as
	\[
	\Mfv(\xi) = I_{\sou} - W(\xi)^{\dagger}.
	\]
	Since $\Asp$ has constant rank, the application $\xi \to \Pi(\xi)$ is smooth away from the origin. Moreover, the rank of $W(\xi)$ is independent of $\xi \in \mathbb{R}^d \mysm \{0\}$: indeed, the image of $W(\xi)$ is equal to $P_{\ker \Afv(\xi)^{\perp}}(V)$, which has the same dimension as $V$ since $\ker \Afv(\xi) \cap V = \{0\}$.
	By \Cref{res:boundPseudo} in \Cref{appx:bound} (based on \cite{golubDifferentiationPseudoInversesNonlinear1973}), there exist constants $(C_{\ell})_{\ell \in \mathbb{N}_*}$ with $C_{\ell}$ depending only on $\sup_{|\alpha| \leqslant \ell} \sup_{|\xi| = 1} \|\partial^{\alpha}_{\xi} W(\xi)\|$ such that
	\begin{align*}
		\|\partial^{\alpha}_{\xi} \Mfv(\xi)\| \leqslant C_{\ell} |\xi|^{-\ell} \|W(\xi)^{\dagger}\|^{\ell+1} \qquad \forall \xi \in \mathbb{R}^d \mysm \{0\} \text{ and } |\alpha| = \ell.
	\end{align*}
	As $\|\partial^{\alpha}_{\xi} W(\xi)\| = \|\partial^{\alpha}_{\xi} \Pi(\xi) P_V\| \leqslant \|\partial^{\alpha}_{\xi} \Pi(\xi)\|$, the constants $C_{\ell}$ are bounded independently of the subspace $V$.
	To estimate $\|W^{\dagger}(\xi)\|$, note that $W^{\dagger}(\xi) v = 0$ whenever $v$ is orthogonal to the image of $W(\xi)$. Hence we can write
	\begin{align*}
		\|W^{\dagger}(\xi)\|
		= \sup_{v \neq 0} \frac{|W^{\dagger}(\xi) v|}{|v|}
		= \sup_{w \neq 0} \frac{|W^{\dagger}(\xi) W(\xi) w|}{|W(\xi) w|}
		= \sup_{w \neq 0} \frac{|\proj_{\mathrm{Im} W^*(\xi)} w|}{|\proj_{\ker \Afv(\xi)^{\perp}} P_V w|}.
	\end{align*}
	Since the image of $W^*$ is contained in $V$, the rightmost term is unchanged if $w$ is replaced by its projection on $V$. As $\ker\Afv(\xi) \subset \Lambda_\Asp$, for any $w \in V \mysm \{0\}$ there holds $|\proj_{\ker \Afv(\xi)^{\perp}} w| = \dist(w, \ker \Afv(\xi)) \ge \dist(w,\Lambda_\Asp) \geqslant |w| d(V,\Lambda_{\Asp})$, whence
	\begin{align*}
		\|W^{\dagger}(\xi)\|
		= \sup_{w \in V \mysm \{0\}} \frac{|\proj_{\mathrm{Im} W^*(\xi)} w|}{|\proj_{\ker \Afv(\xi)^{\perp}} w|}
		\leqslant \sup_{w \in V \mysm \{0\}} \frac{|\proj_{\mathrm{Im} W^*(\xi)} w|}{|w| d(V, \Lambda_{\Asp})}
		= \frac{1}{d(V, \Lambda_{\Asp})}.
	\end{align*}
	Note in passing that $\|W^{\dagger}(\xi)\| \ge 1$: indeed $W(\xi)$ is a product of orthogonal projections, so $\|W(\xi)\| \le 1$, while $W^{\dagger} = W^{\dagger} W W^{\dagger}$ forces $\|W^{\dagger}\| \le \|W^{\dagger}\|^2 \|W\|$.
	Finally, we define the multiplier operator $T_V u = (\Mfv(\cdot) \widehat u)^\vee$. For $\ell = 0$ the displayed bound also holds, since $\Mfv = I_{\sou} - W^{\dagger}$ gives $\|\Mfv(\xi)\| \le 1 + \|W^{\dagger}(\xi)\| \le 2\|W^{\dagger}(\xi)\|$ by the observation above. Given that $\Mfv(\cdot)$ is a smooth zero-homogeneous multiplier and Mikhlin's theorem \cite[Theorem~6.2.7]{grafakosClassicalFourierAnalysis2014}---whose weak type $(1,1)$ bound depends linearly on $\sup_{|\alpha| \le \lfloor d/2\rfloor+1}\sup_{\xi \ne 0}|\xi|^{|\alpha|}\|\partial^\alpha \Mfv(\xi)\|$---only requires to bound the first $\lfloor \frac{d}{2} \rfloor + 1$ derivatives of $\Mfv$, we conclude that
	\[
	\| T_V \|_{L^1 \to L^{1,\infty}}
	\le \sum_{\ell = 0}^{\lfloor \frac{d}{2} \rfloor + 1} \frac{C_\ell}{d(V,\Lambda_{\Asp})^{\ell+1}}
	\le \frac{\ckorn }{d(V,\Lambda_{\Asp})^{\lfloor \frac{d}{2} \rfloor + 2}}
	\]
	for a sufficiently large constant $\ckorn $ depending only on $\Afv$ and the dimension $d$. Since by construction $T_V P_{V^\perp} u = u$ for all $\Asp$-free maps, and $\|P_{V^\perp} u\|_{L^1} = \|\dist(u,V)\|_{L^1}$, we conclude that
	\begin{align}\label{eq:korn:final}
		\| u \|_{L^{1,\infty}}
		= \|T_V P_{V^\perp} u\|_{L^{1,\infty}}
		\le \ckorn \, d(V,\Lambda_{\Asp})^{-\lfloor \frac{d}{2} \rfloor - 2}\,\|\dist(u,V)\|_{L^1}.
	\end{align}
	By definition $\const{V}$ is the optimal constant satisfying \eqref{eq:korn:final}, hence \eqref{eq:korntype:statement} follows.
\end{proof}
\begin{remark}[Optimal $L^p$ constant]
	An analogous statement holds for $p \in (1,\infty)$ and replacing the $L^1$ and $L^{1,\infty}$ norms in the definition of $\const{\cdot}$ with the $L^p$ norm.
\end{remark}
We will need a version of \Cref{res:korntype} on periodic $\Asp$-free functions. The transfer of inequalities from $\R^d$ to the torus is a classical topic in Fourier analysis, see \cite[Section~4.3.2]{grafakosClassicalFourierAnalysis2014}. However in our case, one can also give a short proof based on potentials for constant-rank operators \cite{raitaPotentialsAquasiconvexity2019}, using standard arguments in homogenization \cite{fonsecaAQuasiconvexityLowerSemicontinuity1999}.
Denote $\mathbb{T}^d$ the $d$-dimensional torus, with $\|\cdot\|_{L^p(\mathbb{T}^d)}$ the standard $L^p$ norm on $\mathbb{T}^d$, and
\begin{align*}
	\|u\|_{L^{1,\infty}(\mathbb{T}^d)}
	\coloneqq \sup_{\lambda > 0} \lambda\, \mathscr L^d_{\mathbb{T}^d}(\{|u| > \lambda\}).
\end{align*}
\begin{lemma}\label{res:transfer}
	Let $\const{V}$ be a constant such that $\|u\|_{L^{1,\infty}} \leqslant \const{V} \|\dist(u,V)\|_{L^1}$ for any $u \in \regC^{\infty}_c(\R^d;\sou)$ such that $\Asp u = 0$. Then
	\begin{align*}
		\|w\|_{L^{1,\infty}(\mathbb{T}^d)} \leqslant \const{V} \|\dist(w,V)\|_{L^1(\mathbb{T}^d)}
	\end{align*}
	for any $w \in \regC^{\infty}(\mathbb{T}^d;\sou)$ that is periodic, $\Asp$-free, and with $\int_{\mathbb{T}^d} w \, dx = 0$.
\end{lemma}
The proof is postponed to \Cref{appx:transfer}.
%%  2.2  Integrands and A-quasiconvex functions  -----------------------------
\subsection{Integrands and $\Asp$-quasiconvex functions}\label{sec:Aqcf}

We will be interested only in a class of functions with linear growth and a well-defined recession function, as we now introduce.
Following \cite{alibertNonUniformIntegrabilityGeneralized1997}, define $\predual(\sou)$ as the set of continuous $f : \sou \to \R$ such that the function
\begin{align*}
	\hat{f} : \ball_{\sou} \to \R, \qquad
	\hat{f}(\var) \coloneqq (1-|\var|) f\left(\frac{\var}{1-|\var|}\right)
\end{align*}
extends continuously to an element of $\regC(\overline{\ball}_{\sou})$. In particular, such $f$'s have at most linear growth. $\predual(\sou)$ is a separable Banach space when endowed with $\|f\| \coloneqq \|\hat{f}\|_{\infty}$, and each $f \in \predual(\sou)$ admits a \emph{recession function}, given by
\begin{align*}
	f^{\infty}(\var) \coloneqq \lim_{t \to \infty} \frac{f(t \var)}{t}.
\end{align*}
We refer the reader to \cite[Section~2]{alibertNonUniformIntegrabilityGeneralized1997} and \cite[Section~1.3]{arroyo-rabasaCharacterizationGeneralizedYoung2021} for details on $\predual(\sou)$. Similarly, one defines $\predual(\Omega,\sou)$ as the space of all continuous integrands $f : \Omega \times \sou \to \R$ with uniform linear growth at infinity
\[
\exists C > 0, \qquad |f(x,\var)| \le C(1 + |\var|) \qquad \forall (x,\var) \in \Omega \times \sou
\]
such that the following limit exists
\begin{align*}
	f^{\infty}(x,\var) \coloneqq \lim_{\substack{x' \to x \\ \var' \to \var \\ t \to \infty}} \frac{f(x', t \var')}{t}
\end{align*}
and defines a continuous function on $\overline \Omega \times \sou$.
We recall the concept of quasiconvexity for general $\Asp$-free structures introduced by Fonseca and M\"uller:
\begin{definition}[{\cite{fonsecaAQuasiconvexityLowerSemicontinuity1999}}]\label{def:Aqc}
	A continuous function $f : \sou \to \R$ is called $\Asp$-quasiconvex if for every $\var \in \sou$ it holds
	\[
	f(\var) \le \int_{[0,1]^d} f(\var + u(y)) \, dy
	\]
	for all $\Asp$-free periodic fields $u \in \regC^{\infty}(\mathbb T^d;\sou)$ with $\int_{[0,1]^d} u \, dy = 0$, where $\mathbb T^d = \R^d / \mathbb Z^d$ is the $d$-dimensional flat torus.
\end{definition}
\begin{remark}
	It has been proved in \cite{raitaPotentialsAquasiconvexity2019} that if $\Asp$ has constant rank, then the definition of $\Asp$-quasiconvexity remains unchanged by testing with $\Asp$-free fields $u \in \regC_c^\infty((0,1)^d;\sou)$.
\end{remark}
\begin{definition}
	We denote by $\predualAqc(\sou)$ the set of $\Asp$-quasiconvex functions in $\predual(\sou)$.
\end{definition}
We will use that a continuous $f : \sou \to \R$ is $\Asp$-quasiconvex if and only if it coincides with its $\Asp$-quasiconvex envelope $\QC f : \sou \to \R$, defined as
\begin{align*}
	\QC f(\var) \coloneqq \inf \left\{ \int_{[0,1]^d} f(\var + u(y)) \, dy \ :\ \begin{matrix} u \in \regC^{\infty}(\mathbb T^d;\sou) \\ \text{has zero average, and } \Asp u = 0 \end{matrix} \right\}.
\end{align*}
The following elementary observation is the source of everything in this subsection: an $\Asp$-free periodic field of zero average cannot leave the linear span of the wave cone
\[
	\sou_\Asp \coloneqq \myspan \Lambda_\Asp .
\]
\begin{lemma}\label{res:VAvalued}
	Let $u \in \regC^{\infty}(\mathbb T^d;\sou)$ be $\Asp$-free with $\int_{\mathbb T^d} u \, dy = 0$. Then $u$ takes values in $\sou_\Asp$.
\end{lemma}
\begin{proof}
	For $k \in \mathbb Z^d \mysm \{0\}$ the Fourier coefficients of $u$ satisfy $\Afv(k) \widehat u(k) = 0$. As $\Afv(k)$ has real entries, its complex kernel is the complexification of $\ker \Afv(k)$, and $\ker \Afv(k) = \ker \Afv(k/|k|) \subset \Lambda_\Asp \subset \sou_\Asp$ by $k$-homogeneity of the symbol. Hence $\widehat u(k) \in \sou_\Asp \otimes \mathbb C$ for every $k \neq 0$, while $\widehat u(0) = 0$; summing the Fourier series gives $u(y) \in \sou_\Asp$ for every $y$.
\end{proof}
We next record that $\Asp$-quasiconvex functions of linear growth are Lipschitz along the directions of $\sou_\Asp$, \emph{uniformly with respect to the base point}.
\begin{lemma}\label{res:coneLip}
	Let $f : \sou \to \R$ be continuous, $\Asp$-quasiconvex and of linear growth at infinity. Then there exists $C = C(\Asp)$, depending also on the growth constant of $f$, such that
	\begin{equation}\label{eq:coneLip}
		|f(\varalt + \var) - f(\varalt + \var')| \le C |\var - \var'|
		\qquad \forall\, \varalt \in \sou, \quad \forall\, \var, \var' \in \sou_\Asp .
	\end{equation}
\end{lemma}
\begin{proof}
	For $\zeta \in \Lambda_\Asp$ and $\varalt \in \sou$ the function $t \mapsto f(\varalt + t\zeta)$ is convex \cite[Proposition~3.4]{fonsecaAQuasiconvexityLowerSemicontinuity1999} and of linear growth, hence Lipschitz with a constant proportional to $|\zeta|$ and independent of $\varalt$. Since $\Lambda_\Asp$ spans $\sou_\Asp$ and $-\Lambda_\Asp = \Lambda_\Asp$, we may fix a basis $\zeta_1,\dots,\zeta_n$ of $\sou_\Asp$ contained in $\Lambda_\Asp$; every $\var - \var' \in \sou_\Asp$ then decomposes as $\sum_i a_i\zeta_i$ with $\sum_i |a_i| |\zeta_i| \le K|\var - \var'|$, $K$ depending only on that basis, and \eqref{eq:coneLip} follows by chaining the $n$ one-dimensional estimates. See also \cite[Prop.~4.5]{arroyo-rabasaCharacterizationGeneralizedYoung2021}.
\end{proof}
Applying the same observation to the \emph{complementary} projection is even more useful: an integrand that depends only on the component transversal to $\sou_\Asp$ is $\Asp$-quasiconvex \emph{together with its opposite}, so that testing a Jensen inequality against it produces an \emph{equality}.
\begin{lemma}\label{res:transverse}
	Let $\Asp$ satisfy the constant-rank \Cref{ass:constant_rank}, write $Q \coloneqq P_{\sou_\Asp^{\perp}}$, and let $g : \sou_\Asp^{\perp} \to \R$ be continuous. If $g \circ Q \in \predual(\sou)$, then both $g \circ Q$ and $-\,g \circ Q$ belong to $\predualAqc(\sou)$, with $(-\,g\circ Q)^{\infty} = -\,(g\circ Q)^{\infty}$. Consequently $f + g \circ Q \in \predualAqc(\sou)$ for every $f \in \predualAqc(\sou)$.
\end{lemma}
\begin{proof}
	Let $u \in \regC^{\infty}(\mathbb T^d;\sou)$ be $\Asp$-free with zero average. By \Cref{res:VAvalued} the field $u$ takes values in $\sou_\Asp$, so that $Qu \equiv 0$ and
	\[
	\int_{[0,1]^d} g\big(Q(\var + u(y))\big) \, dy = \int_{[0,1]^d} g(Q\var) \, dy = g(Q\var)
	\qquad \forall\, \var \in \sou .
	\]
	Thus $g \circ Q$ satisfies the defining inequality of \Cref{def:Aqc} with equality, and so does $-\,g\circ Q$; the identity between the recession functions is immediate from their definition. The last assertion follows by adding that equality to the inequality defining the $\Asp$-quasiconvexity of $f$, the space $\predual(\sou)$ being linear.
\end{proof}

The following argument is extracted from the proof of \cite[Thm.~1]{kristensenConcentrationEffectsBV2021}. We need a version that does not presuppose that $\Lambda_\Asp$ spans $\sou$. It turns out that no extra hypothesis is needed for this: the Jensen inequality \eqref{KR:hyp} \emph{by itself} forces $\nu$ and $\nu^{\infty}$ to be carried by the appropriate cosets of $\sou_\Asp$. This observation is due to Arroyo-Rabasa \cite[Cor.~4.1]{arroyo-rabasaCharacterizationGeneralizedYoung2021}; we reproduce it as Step~0 below, and it is what takes over the r\^ole of the spanning assumption.
\begin{lemma}[Extended Kristensen--\Raita decoupling]\label{res:KR}
	Let $\nu \in \prob(\sou)$ have finite first moment and let $\nu^{\infty} \in \prob(\sphere_{\sou})$. Suppose that there exists $\lambda > 0$ such that
	\begin{align}\label{KR:hyp}
		f(\bar \nu + \lambda \bar \nu^{\infty}) \le \dpr{f, \nu} + \lambda \dpr{f^{\infty}, \nu^{\infty}}
	\end{align}
	for any $f \in \predualAqc(\sou)$. Then
	\begin{align}\label{KR:statement}
		f(\var + \bar \nu^{\infty}) \le f(\var) + \dpr{f^{\infty}, \nu^{\infty}}
		\qquad \forall f \in \predualAqc(\sou) \text{ and } \var \in \sou.
	\end{align}
\end{lemma}
\begin{proof}
	Throughout, $C$ denotes the constant of \eqref{eq:coneLip}.
	\emph{Step 0: the support constraints are automatic.} We claim that \eqref{KR:hyp} alone implies
	\begin{align}\label{KR:supp}
		\mathrm{spt}\, \nu \subset \bar\nu + \sou_\Asp
		\qquad \text{and} \qquad
		\mathrm{spt}\, \nu^{\infty} \subset \sou_\Asp .
	\end{align}
	Write $Q \coloneqq P_{\sou_\Asp^{\perp}}$ and $\pol \coloneqq \bar\nu + \lambda\bar\nu^{\infty}$. Let first $g \in \regC_c(\sou_\Asp^{\perp})$. Then $g \circ Q$ is bounded and continuous, hence lies in $\predual(\sou)$ with $(g \circ Q)^{\infty} \equiv 0$, and \Cref{res:transverse} puts both $\pm\, g\circ Q$ in $\predualAqc(\sou)$. Applying \eqref{KR:hyp} to each sign turns the inequality into an equality,
	\[
	g(Q\pol) \ =\ \dpr{g \circ Q, \nu} \ =\ \dpr{g, Q_{\#}\nu}
	\qquad \forall\, g \in \regC_c(\sou_\Asp^{\perp}),
	\]
	so that $Q_{\#}\nu = \delta_{Q\pol}$, because $\regC_c(\sou_\Asp^{\perp})$ separates the probability measures on $\sou_\Asp^{\perp}$. Hence $Q\var = Q\pol$ for $\nu$-almost every $\var$, that is, $\mathrm{spt}\,\nu \subset \pol + \sou_\Asp$; integrating this identity against $\nu$ gives $Q\bar\nu = Q\pol$, whence $\lambda\, Q\bar\nu^{\infty} = 0$ and, $\lambda$ being positive, $\pol + \sou_\Asp = \bar\nu + \sou_\Asp$. This is the first constraint.
	Take now $g \coloneqq |\,\cdot\,|$ on $\sou_\Asp^{\perp}$, so that $g \circ Q = |Q(\,\cdot\,)|$ is continuous and positively $1$-homogeneous; it therefore lies in $\predual(\sou)$ with $(g\circ Q)^{\infty} = g \circ Q$, and again $\pm\, g \circ Q \in \predualAqc(\sou)$. Both signs of \eqref{KR:hyp} now give
	\[
	|Q\pol| \ =\ \int_{\sou} |Q\var| \, d\nu(\var) \ +\ \lambda \int_{\sphere_{\sou}} |Q\var| \, d\nu^{\infty}(\var),
	\]
	all three terms being finite because $\nu$ has finite first moment. By the first constraint the first integral on the right equals $|Q\pol|$, so the second vanishes; as $\lambda > 0$, this is the second constraint. In particular $\sou_\Asp \cap \sphere_{\sou} \neq \emptyset$, so that $\Lambda_\Asp \neq \{0\}$.
	\emph{Step 1: Reduction to a centred measure carried by $\sou_\Asp$.}
	$\Asp$-quasiconvexity, membership in $\predual(\sou)$ and the recession function are all unaffected by a translation of the argument: if $f \in \predualAqc(\sou)$ then $f(\bar\nu + \,\cdot\,) \in \predualAqc(\sou)$ and $f(\bar\nu + \,\cdot\,)^{\infty} = f^{\infty}$. Replacing simultaneously $f$ by $f(\bar\nu + \,\cdot\,)$ and $\nu$ by $(\,\cdot - \bar\nu)_{\#}\nu$ therefore leaves \eqref{KR:hyp} unchanged, while the conclusion \eqref{KR:statement} does not involve $\nu$ at all. In view of \eqref{KR:supp} we may hence assume that
	\[
	\bar\nu = 0 \qquad \text{and} \qquad \mathrm{spt}\,\nu \subset \sou_\Asp .
	\]
	\emph{Step 2: Differentiation.}
	By \eqref{eq:coneLip} the restriction of $f$ to any coset $\varalt + \sou_\Asp$ is Lipschitz with constant $C$. Rademacher's theorem, applied on $\sou_\Asp$ and combined with Fubini's theorem in the splitting $\sou = \sou_\Asp^{\perp} \oplus \sou_\Asp$, therefore shows that $G \coloneqq \big\{ \varalt \in \sou \st f|_{\varalt + \sou_\Asp} \text{ is differentiable at } \varalt \big\}$ has full $\mathscr L^{N}$-measure in $\sou$. For $\varalt \in G$ let $\nabla_\Asp f(\varalt)$ denote the differential of that restriction, so that
	\[
	\lim_{h \to 0^+} \frac{f(\varalt + h\var) - f(\varalt)}{h} = \dpr{\nabla_\Asp f(\varalt), \var}
	\qquad \text{for every } \var \in \sou_\Asp .
	\]
	Fix $\varalt \in G$ and $h > 0$, and set $\phi_h(\var) \coloneqq h^{-1}\big(f(\varalt + h\var) - f(\varalt)\big)$. Then $\phi_h \in \predualAqc(\sou)$, since $\Asp$-quasiconvexity is preserved by the affine change of variables $\var \mapsto \varalt + h\var$ and by subtraction of a constant, and $\phi_h^{\infty} = f^{\infty}$. Applying \eqref{KR:hyp} to $\phi_h$ and recalling $\bar\nu = 0$,
	\begin{align*}
		\frac{f(\varalt + h\lambda\bar\nu^{\infty}) - f(\varalt)}{h}
		\ \le\ \int_{\sou} \frac{f(\varalt + h\var) - f(\varalt)}{h} \, d\nu(\var)
		+ \lambda \dpr{f^{\infty}, \nu^{\infty}} .
	\end{align*}
	Both $\bar\nu^{\infty}$ and $\nu$-almost every $\var$ lie in $\sou_\Asp$, so \eqref{eq:coneLip} bounds the integrand by $C|\var| \in L^1(\nu)$, while it converges pointwise to $\dpr{\nabla_\Asp f(\varalt),\var}$. Dominated convergence therefore yields
	\[
	\lambda \dpr{\nabla_\Asp f(\varalt), \bar\nu^{\infty}}
	\ \le\ \dpr{\nabla_\Asp f(\varalt), \bar\nu} + \lambda\dpr{f^{\infty},\nu^{\infty}}
	\ =\ \lambda \dpr{f^{\infty}, \nu^{\infty}},
	\]
	and after dividing by $\lambda > 0$,
	\begin{align}\label{proof:KR:ineq}
		\dpr{\nabla_\Asp f(\varalt), \bar\nu^{\infty}} \leqslant \dpr{f^{\infty}, \nu^{\infty}}
		\qquad \text{for every } \varalt \in G .
	\end{align}
	\emph{Step 3: Integration along $\bar\nu^{\infty}$.}
	If $\bar\nu^{\infty} = 0$ then \eqref{KR:statement} is trivial, so assume $\bar\nu^{\infty} \neq 0$. By the second constraint in \eqref{KR:supp}, $\bar\nu^{\infty} \in \sou_\Asp$. Split $\sou = \R\bar\nu^{\infty} \oplus H$. Since $\sou \mysm G$ is Lebesgue-null, Fubini's theorem provides $H_0 \subset H$ of full measure in $H$ such that, for every $y \in H_0$, the line $y + \R\bar\nu^{\infty}$ meets $\sou \mysm G$ in an $\mathscr L^1$-null set. Fix such a $y$ and let $\var \in y + \R\bar\nu^{\infty}$. By \eqref{eq:coneLip} the map $t \mapsto f(\var + t\bar\nu^{\infty})$ is Lipschitz on $[0,1]$, hence absolutely continuous, and for $\mathscr L^1$-a.e.\ $t$ its derivative equals $\dpr{\nabla_\Asp f(\var + t\bar\nu^{\infty}), \bar\nu^{\infty}}$. By \eqref{proof:KR:ineq},
	\begin{align*}
		f(\var + \bar\nu^{\infty}) - f(\var)
		= \int_{0}^{1} \dpr{\nabla_\Asp f(\var + t\bar\nu^{\infty}), \bar\nu^{\infty}} \, dt
		\le \dpr{f^{\infty}, \nu^{\infty}} .
	\end{align*}
	This holds for every $\var$ in $H_0 + \R\bar\nu^{\infty}$, a set of full measure in $\sou$. Finally, since $f$ is continuous, the inequality extends to every $\var \in \sou$.
\end{proof}

It is convenient to give a name to the measures singled out by \eqref{KR:statement}.
\begin{definition}\label{def:micro}
	We write $\Env_\Asp$ for the set of \emph{admissible concentration profiles}, that is, the set of $\nu^{\infty} \in \prob(\sphere_{\sou})$ such that $\mathrm{spt}\,\nu^{\infty} \subset \sou_\Asp$ and
	\begin{align}\label{eq:micro}
		f(\var + \bar\nu^{\infty}) \le f(\var) + \dpr{f^{\infty},\nu^{\infty}}
		\qquad \forall f \in \predualAqc(\sou), \ \forall \var \in \sou .
	\end{align}
\end{definition}
\begin{corollary}\label{res:convexity}
	$\Env_\Asp$ is convex and weak-$*$ compact.
\end{corollary}
\begin{proof}
	Since $f$, $f^{\infty}$ and $\nu^{\infty} \mapsto \bar\nu^{\infty}$ are weak-$*$ continuous, and since $\{\nu^{\infty} : \mathrm{spt}\,\nu^{\infty} \subset \sou_\Asp\}$ is weak-$*$ closed, $\Env_\Asp$ is a weak-$*$ closed subset of $\prob(\sphere_{\sou})$, hence weak-$*$ compact.
	For convexity, let $\nu^{\infty}_0,\nu^{\infty}_1 \in \Env_\Asp$, let $t \in [0,1]$ and set $\nu^{\infty}_t \coloneqq (1-t)\nu^{\infty}_0 + t\nu^{\infty}_1$; its support is contained in $\sou_\Asp$. Fix $f \in \predualAqc(\sou)$ and let $G$ be the set of the proof of \Cref{res:KR}, which depends on $f$ alone. For $i \in \{0,1\}$, evaluating \eqref{eq:micro} at $\var = 0$ shows that the triple $(\delta_0, 1, \nu^{\infty}_i)$ satisfies \eqref{KR:hyp}, while $\bar\delta_0 = 0$ and $\mathrm{spt}\,\delta_0 = \{0\} \subset \sou_\Asp$; Step~2 of the proof of \Cref{res:KR} therefore gives
	\[
	\dpr{\nabla_\Asp f(\varalt), \bar\nu^{\infty}_i} \le \dpr{f^{\infty},\nu^{\infty}_i}
	\qquad \text{for every } \varalt \in G .
	\]
	Taking the convex combination with weight $t$ yields $\dpr{\nabla_\Asp f(\varalt), \bar\nu^{\infty}_t} \le \dpr{f^{\infty},\nu^{\infty}_t}$ on $G$, which is \eqref{proof:KR:ineq} for $\nu^{\infty}_t$. Step~3 of the same proof then gives \eqref{eq:micro} for $\nu^{\infty}_t$, so $\nu^{\infty}_t \in \Env_\Asp$.
\end{proof}
%%  2.3  Young measures  -----------------------------------------------------
\subsection{Young measures}\label{sec:YM}
We will make use of the well-established theory of $\Asp$-free generalized Young measures in the proof of our main result; more precisely, we rely on the Jensen inequalities that arise as necessary conditions for the lower semicontinuity of integrals defined on $\Asp$-free sequences. While our primary focus is on compensated compactness, this measure-theoretic framework provides the precise duality pairing needed to exploit $\Asp$-quasiconvexity.
\begin{definition}[{Generalized Young measures \cite{dipernaOscillationsConcentrationsWeak1987,alibertNonUniformIntegrabilityGeneralized1997}}]\label{def:GYM}
	Let $\Omega \subset \mathbb{R}^d$ be open and let $\sou$ be a finite-dimensional inner product space.
	A generalized Young measure on $\Omega$ with values in $\sou$ is a triple $\boldsymbol{\nu} \coloneqq (\nu,\lambda,\nu^\infty)$, consisting of:
	\begin{enumerate}[label={(\roman*)},itemsep=0pt]
		\item a weak-$*$ measurable family $\nu : x \in \Omega \mapsto \nu_x$ of probability measures on $\sou$,
		\item a non-negative measure $\lambda \in \Meas_+(\overline{\Omega})$,
		\item a weak-$*$ $\lambda$-measurable family $\nu^{\infty} : x \in \Omega \mapsto \nu_x^\infty$ of probability measures over $\sphere_{\sou}$.
	\end{enumerate}
	The collection of such triples $\boldsymbol{\nu}$ satisfying $x \mapsto \int |\var| \, d\nu_x(\var) \in L^1(\Omega)$ and $\lambda(\Omega) < \infty$ is denoted by $\mathbf{Y}(\Omega,\sou)$.
\end{definition}
A Young measure $\boldsymbol{\nu}$ is said to be \emph{homogeneous} if both parameterized families of measures $x \mapsto \nu_x$ and $x \mapsto \nu_x^{\infty}$ are constant. The space of generalized Young measures $\mathbf{Y}(\Omega,\sou)$ can be identified as a subset of the dual of $\predual(\Omega,\sou)$ (defined in \Cref{sec:Aqcf}) via the duality product:
\begin{align*}
	\dpr{f, \boldsymbol{\nu}}
	\coloneqq \int_{\Omega} \dpr{f(x,\cdot), \nu_x} dx + \int_{\overline \Omega} \dpr{f^{\infty}(x,\cdot), \nu_x^{\infty}} d\lambda(x).
\end{align*}
Accordingly, we define the sequential weak-$*$ convergence $\boldsymbol \nu_j \toweakstar \boldsymbol{\nu}$ in $\mathbf Y(\Omega,\sou)$ by the condition that $\dpr{f, \boldsymbol{\nu_j}} \to \dpr{f, \boldsymbol{\nu}}$ for all $f \in \predual(\Omega,\sou)$.
For a given Young measure, we recall that the macroscopic averages, or \emph{barycenters}, of its oscillation and concentration components are respectively defined as
\[
\bar \nu_x \coloneqq \int_{\sou} \var \, d\nu_x(\var)
\qquad \text{and} \qquad
\bar \nu_x^{\infty} \coloneqq \int_{\sphere_{\sou}} \var \, d\nu_x^{\infty}(\var).
\]
We are particularly concerned with Young measures arising from sequences that satisfy the PDE constraint $\Asp u = 0$.
\begin{definition}[$\Asp$-free generalized Young measures]\label{def:A_free_GYM}
	An \emph{$\Asp$-free generalized Young measure} on $\Omega$ is a generalized Young measure $\boldsymbol{\nu} \in \mathbf Y(\Omega,\sou)$ for which there exists a sequence of vector-valued measures $(\mu_j)_j \subset \Meas(\Omega,\sou)$ such that
	\begin{align*}
		\Asp \mu_j = 0 \quad \text{in the sense of distributions on } \Omega,
	\end{align*}
	and for any $f \in \predual(\Omega,\sou)$, there holds
	\begin{align*}
		\int_{\Omega} f(x, \mu_j^{a}(x)) \, dx + \int_{\Omega} f^{\infty}\Big(x,\frac{\mu_j}{|\mu_j|}(x)\Big) \, d|\mu_j|^s
		\underset{j \to \infty}{\longrightarrow}
		\dpr{f, \boldsymbol{\nu}}.
	\end{align*}
	The set of all such $\Asp$-free Young measures is denoted by $\mathbf Y_\Asp(\Omega)$.
\end{definition}
A fundamental property of any $\boldsymbol{\nu} \in \mathbf Y_\Asp(\Omega)$ is that it satisfies the macroscopic barycenter constraint
\begin{align}\label{charGAFYM:bary}
	\Asp \bar{\boldsymbol{\nu}} = 0 \qquad \text{in the sense of distributions on } \Omega,
\end{align}
where the vector-valued measure $\bar{\boldsymbol{\nu}}$ is defined by $\bar{\boldsymbol{\nu}} \coloneqq \bar \nu_x \, dx + \bar{\nu}^\infty_x \, \lambda(dx)$. This is a necessary condition, as the global barycenter captures the weak-$*$ measure limit of the generating sequence. Throughout, $\bar{\boldsymbol{\nu}}^{ac}$ denotes the density of the absolutely continuous part of $\bar{\boldsymbol{\nu}}$ with respect to $\mathscr L^d$, that is, $\bar{\boldsymbol{\nu}}^{ac}(x) = \bar\nu_x + \lambda^{a}(x)\, \bar\nu^{\infty}_x$ for $\mathscr L^d$-almost every $x$.
Any $\Asp$-free Young measure $\boldsymbol{\nu} \in \mathbf Y_\Asp(\Omega)$ furthermore satisfies the algebraic support constraint
\begin{align}\label{eq:support}
	\mathrm{spt}\, \nu_x^\infty \subset \sou_\Asp \quad \lambda\text{-a.e.},
\end{align}
together with the localized Jensen-type inequality
\begin{align}\label{charGAFYM:jensen}
	f\big( \bar{\boldsymbol{\nu}}^{ac}(x) \big)
	\le \dpr{f, \nu_x} + \lambda^{a}(x)\, \dpr{f^{\infty}, \nu_x^\infty} \qquad \mathscr L^d\text{-a.e.}
\end{align}
for all $f \in \predualAqc(\sou)$; see \cite[Thm.~1.1, Rmk.~1.2 and Cor.~4.1]{arroyo-rabasaCharacterizationGeneralizedYoung2021} and \cite{arroyo-rabasaLowerSemicontinuityRelaxation2020}.
\begin{remark}[The hypotheses of \Cref{res:KR}]\label{rem:KRsupp}
	At $\mathscr L^d$-almost every $x$ with $\lambda^{a}(x) > 0$, the inequality \eqref{charGAFYM:jensen} is exactly \eqref{KR:hyp} for the triple $(\nu_x, \lambda^{a}(x), \nu^{\infty}_x)$, which has finite first moment. Nothing else has to be checked --- in particular not \eqref{eq:support}, since by Step~0 of the proof of \Cref{res:KR} the support constraints are themselves a consequence of \eqref{charGAFYM:jensen}. \Cref{res:KR} therefore applies at every such point and yields $\nu^{\infty}_x \in \Env_\Asp$.
\end{remark}
In fact, it was shown in \cite{arroyo-rabasaCharacterizationGeneralizedYoung2021} (see also \cite{kristensenOscillationConcentrationSequences2022}) that conditions \eqref{charGAFYM:bary}--\eqref{charGAFYM:jensen} are also sufficient to fully characterize $\Asp$-free Young measures among all elements of $\mathbf{Y}(\Omega,\sou)$. Combined with \Cref{res:KR}, this identifies $\Env_\Asp$ exactly:
\begin{proposition}[Concentration profiles]\label{prop:micro_Young}
	Let $\Asp$ satisfy the constant-rank \Cref{ass:constant_rank} and let $\nu^{\infty} \in \prob(\sphere_{\sou})$. The following are equivalent:
	\begin{enumerate}[label={(\roman*)}]
		\item\label{item:micro1} $\boldsymbol{\nu} = (\nu,\lambda,\nu^{\infty}) \in \mathbf Y_\Asp(\Omega)$ is a homogeneous Young measure for some $\lambda \in \Meas_+(\overline\Omega)$ with $\lambda(\Omega) > 0$;
		\item\label{item:micro2} $\nu^{\infty} \in \Env_\Asp$.
	\end{enumerate}
\end{proposition}
\begin{proof}
	Assume \ref{item:micro1}. The support constraint $\mathrm{spt}\,\nu^{\infty} \subset \sou_\Asp$ is part of \eqref{eq:support}, so only \eqref{eq:micro} has to be established. We distinguish two cases according to whether the absolutely continuous part of $\lambda$ vanishes.
	\emph{Case 1: $\lambda^{a} \equiv 0$.} Then $\lambda = \lambda^{s}$, and since $\nu^{\infty}$ is constant in $x$ the barycenter measure of \eqref{charGAFYM:bary} decomposes as
	\[
	\bar{\boldsymbol{\nu}} = \bar\nu \, \mathscr L^d + \bar\nu^{\infty} \lambda,
	\qquad \text{with} \qquad
	\bar{\boldsymbol{\nu}}^{s} = \bar\nu^{\infty}\lambda .
	\]
	We claim that $\bar\nu^{\infty} \in \Lambda_\Asp$. This is trivial if $\bar\nu^{\infty} = 0$, because $0 \in \ker \Afv(\xi) \subset \Lambda_\Asp$. If $\bar\nu^{\infty} \neq 0$, then $|\bar{\boldsymbol{\nu}}|^{s} = |\bar\nu^{\infty}|\lambda$ is non-trivial on $\Omega$ and its polar is the constant vector $\bar\nu^{\infty}/|\bar\nu^{\infty}|$; as $\bar{\boldsymbol{\nu}}$ is $\Asp$-free by \eqref{charGAFYM:bary}, \Cref{thm:DPR} forces $\bar\nu^{\infty}/|\bar\nu^{\infty}| \in \Lambda_\Asp$, and the claim follows because $\Lambda_\Asp$ is a cone.
	Fix now $f \in \predualAqc(\sou)$ and $\var \in \sou$. Since $f$ is $\Lambda_\Asp$-convex \cite[Proposition~3.4]{fonsecaAQuasiconvexityLowerSemicontinuity1999} and $\bar\nu^{\infty} \in \Lambda_\Asp$, the function $t \mapsto f(\var + t\bar\nu^{\infty})$ is convex; its difference quotients are therefore non-decreasing and
	\begin{align}\label{eq:micro:subadd}
		f(\var + \bar\nu^{\infty}) - f(\var)
		\ \le\ \lim_{t \to \infty} \frac{f(\var + t \bar\nu^{\infty}) - f(\var)}{t}
		\ =\ f^{\infty}(\bar\nu^{\infty}).
	\end{align}
	On the other hand, by \cite[Theorem~1.1]{kirchheimRankOneConvex2016} the positively $1$-homogeneous function $f^{\infty}$ has a non-empty convex subdifferential at every point of $\Lambda_\Asp$; picking $\zeta$ in the subdifferential at $\bar\nu^{\infty}$ and integrating
	\[
	f^{\infty}(\varalt) \ \ge\ f^{\infty}(\bar\nu^{\infty}) + \dpr{\zeta, \varalt - \bar\nu^{\infty}}
	\qquad (\varalt \in \sou)
	\]
	against $\nu^{\infty}$ gives $\dpr{f^{\infty},\nu^{\infty}} \ge f^{\infty}(\bar\nu^{\infty})$. Combining this with \eqref{eq:micro:subadd} yields \eqref{eq:micro}, and hence $\nu^{\infty} \in \Env_\Asp$.
	\emph{Case 2: $\lambda^{a} \not\equiv 0$.} Since $\predual(\sou)$ is separable, so is $\predualAqc(\sou)$; let $D$ be a countable dense subset of it. Both sides of \eqref{charGAFYM:jensen} depend continuously on $f$ with respect to $\|\cdot\| = \|\hat{\,\cdot\,}\|_{\infty}$, because $|g(\varalt)| \le \|g\|(1+|\varalt|)$ for every $g \in \predual(\sou)$, because $\nu$ has finite first moment, and because $g^{\infty}$ is the restriction of $\hat g$ to $\sphere_{\sou}$, so that $\|g^{\infty}\|_{\infty} \le \|g\|$. Discarding the countably many null sets attached to the elements of $D$ and passing to the limit, we find a single $\mathscr L^d$-null set outside of which \eqref{charGAFYM:jensen} holds simultaneously for every $f \in \predualAqc(\sou)$. As $\mathscr L^d(\{\lambda^{a} > 0\}) > 0$, we may pick a point $x_0$ outside that null set with $\lambda_0 \coloneqq \lambda^{a}(x_0) > 0$. Because $\nu$ and $\nu^{\infty}$ do not depend on $x$, \eqref{charGAFYM:jensen} at $x_0$ reads
	\[
	f(\bar\nu + \lambda_0 \bar\nu^{\infty}) \le \dpr{f,\nu} + \lambda_0 \dpr{f^{\infty},\nu^{\infty}}
	\qquad \text{for all } f \in \predualAqc(\sou),
	\]
	which is precisely \eqref{KR:hyp} for the triple $(\nu,\lambda_0,\nu^{\infty})$. As this is its only hypothesis, \Cref{res:KR} yields \eqref{eq:micro}, that is, $\nu^{\infty} \in \Env_\Asp$.
	Conversely, assume \ref{item:micro2} and consider the homogeneous triple $\boldsymbol{\nu} \coloneqq (\delta_0,\, \mathscr L^d \resmes \Omega,\, \nu^{\infty})$, for which $\bar{\boldsymbol \nu}^{ac}(x) \equiv \bar\nu^{\infty}$. Its barycenter $\bar{\boldsymbol{\nu}} = \bar\nu^{\infty} \mathscr L^d \resmes \Omega$ is a constant multiple of the Lebesgue measure, so \eqref{charGAFYM:bary} holds because $\Asp$ annihilates constants; \eqref{eq:support} holds since $\mathrm{spt}\,\nu^{\infty} \subset \sou_\Asp$; and \eqref{charGAFYM:jensen} reads $f(\bar\nu^{\infty}) \le f(0) + \dpr{f^{\infty},\nu^{\infty}}$, which is \eqref{eq:micro} at $\var = 0$. By the sufficiency of \eqref{charGAFYM:bary}--\eqref{charGAFYM:jensen}, $\boldsymbol{\nu} \in \mathbf Y_\Asp(\Omega)$, and $\lambda(\Omega) = \mathscr L^d(\Omega) > 0$. This is \ref{item:micro1}.
\end{proof}
We may now prove the statement announced in the introduction.
\begin{proof}[Proof of \Cref{prop:convexity}]
	Denote by $\mathscr P$ the set of all $\nu^{\infty}_{x}$ arising as the angular concentration measure, at $\lambda$-almost every point $x$, of some $\boldsymbol\nu = (\nu,\lambda,\nu^{\infty}) \in \mathbf Y_\Asp(\Omega)$. We claim that $\mathscr P = \Env_\Asp$; the assertion then follows from \Cref{res:convexity}.
	The inclusion $\Env_\Asp \subset \mathscr P$ is the implication \ref{item:micro2} $\Rightarrow$ \ref{item:micro1} of \Cref{prop:micro_Young}. For the converse \ref{item:micro1} $\Rightarrow$ \ref{item:micro2} we invoke the localization principle for $\Asp$-free generalized Young measures of Arroyo-Rabasa, De Philippis and Rindler \cite[Section~4]{arroyo-rabasaLowerSemicontinuityRelaxation2020} (see also \cite[Section~4]{arroyo-rabasaCharacterizationGeneralizedYoung2021}): blowing up $\boldsymbol\nu$ at a point $x$, at $\mathscr L^d$-almost every regular point and at $\lambda^{s}$-almost every singular point one obtains a \emph{homogeneous} tangent Young measure $\boldsymbol\nu^{x} \in \mathbf Y_\Asp(\ball_1)$ whose oscillation and concentration measures are $\nu_x$ and $\nu^{\infty}_x$ and whose $\lambda$ is a non-trivial multiple of $\mathscr L^d \resmes \ball_1$ at regular points, respectively a non-trivial positive measure at singular points. In either case $\boldsymbol\nu^{x}$ satisfies \ref{item:micro1}, so that $\nu^{\infty}_x \in \Env_\Asp$ by \Cref{prop:micro_Young}.
\end{proof}
%%%%%%%%%%%%%%%%%%%%%%%%%%%%%%%%%%%%%%%%%%%%%%%%%%%%%%%%%%%%%%%%%%%%%%%%%%%%%%
%%  3.  CONSTRUCTIONS AND THE SPREADING INEQUALITY
%%%%%%%%%%%%%%%%%%%%%%%%%%%%%%%%%%%%%%%%%%%%%%%%%%%%%%%%%%%%%%%%%%%%%%%%%%%%%%
\section{Constructions and the spreading inequality}\label{sec:constructions}
The main new tool is an inequality implied by Jensen inequalities against $\Asp$-quasiconvex functions, taking inspiration from the construction of \cite{mullerQuasiconvexFunctionsWhich1992}.
%%  3.1  A spreading inequality for A-free concentrations  -------------------
\subsection{A spreading inequality for $\Asp$-free concentrations}
The following lemma establishes the explicit construction of the geometric penalty used in \Cref{sec:dpr}, and states that a measure satisfying such Jensen inequalities cannot concentrate arbitrarily close to an elliptic subspace if its barycenter is bounded away from the wave cone.
\begin{lemma}\label{res:spreading}
	Assume that $\Asp$ has constant rank and let $V \subset \sou$ be a linear subspace. Then,
	\begin{enumerate}[label={(\alph*)}]
		\item\label{item:spreading:exif} $V \cap \Lambda_{\Asp} = \{0\}$ if and only if for any given $X_0 \in V \smallsetminus \{0\}$ there exists $f \in \predualAqc(\sou)$ such that
		\begin{align*}
			f(X_0) > f(0) + f^{\infty}(X_0).
		\end{align*}
		Moreover, $f$ can be chosen to be Lipschitz.
		\item\label{item:spreading:ineq} Let $\nu^{\infty} \in \prob(\sphere_{\sou})$ be supported on $\sou_\Asp$ and suppose that
		\begin{align}\label{spreading:statement:jensen}
			f(\pol) \le f(0) + \dpr{f^{\infty}, \nu^{\infty}}
			\quad \text{for all } f \in \predualAqc(\sou),
		\end{align}
		where $\bar \nu^{\infty} = \pol$. If $V \cap \Lambda_{\Asp} = \{0\}$, then
		\begin{align}\label{spreading:statement:ineq}
			|P_V \pol| - (1+\const{V}) |P_{V^\perp} \pol| \le \const{V} \int_{\sphere_{\sou}} \dist(\var, V) \, d \nu^{\infty}(\var).
		\end{align}
		Here $\const{V}$ is the optimal Korn constant from \Cref{res:korntype}.
	\end{enumerate}
\end{lemma}
\begin{proof} The ``if'' implication in~\ref{item:spreading:exif} -- that no such $f$ can exist when $X_0\in\Lambda_{\Asp}$ -- follows directly from~\cite{kirchheimRankOneConvex2016}. We may hence focus on the ``only if'' implication, i.e.\ on constructing $f$ when $V\cap\Lambda_{\Asp}=\{0\}$.
	First, we observe that the proof of~\ref{item:spreading:exif} when $X_0 \in \sou \smallsetminus \sou_\Asp$ is trivial. Indeed, write $Q \coloneqq P_{\sou_\Asp^{\perp}}$; then $Q\pol \neq 0$, so we may pick a Lipschitz $g \in \regC_c(\sou_\Asp^{\perp})$ with $g(Q\pol) > 0 = g(0)$. By \Cref{res:transverse} the integrand $f \coloneqq g \circ Q$ lies in $\predualAqc(\sou)$, it is Lipschitz, and $f^{\infty} \equiv 0$ because $g$ is bounded; hence $f(\pol) > 0 = f(0) + f^{\infty}(\pol)$. More generally, \Cref{res:transverse} says that any $f \in \predualAqc(\sou)$ may be modified by an arbitrary function of the transversal component without losing $\Asp$-quasiconvexity. Secondly, the proof of~\ref{item:spreading:ineq} is trivial for $\pol = 0$. We may henceforth assume that $\pol \in \sou_{\Asp} \smallsetminus \{0\}$.
	In proving \ref{item:spreading:ineq} we may moreover assume that
	\begin{align}\label{spreading:proof:nondeg}
		|P_V \pol| \ >\ |P_{V^{\perp}} \pol| ,
	\end{align}
	because otherwise the left-hand side of \eqref{spreading:statement:ineq} is at most $|P_{V^{\perp}}\pol| - (1+\const{V})|P_{V^{\perp}}\pol| = -\const{V}|P_{V^{\perp}}\pol| \le 0$, whereas its right-hand side is non-negative, and there is nothing to prove. In the situation of \ref{item:spreading:exif} the condition \eqref{spreading:proof:nondeg} holds automatically, since there $\pol \in V \mysm \{0\}$.

	The proof of both \ref{item:spreading:exif} and \ref{item:spreading:ineq} will be a direct consequence of the following construction, which is inspired by the strategy in \cite{mullerQuasiconvexFunctionsWhich1992} to build non-convex quasiconvex functions.
	To reduce notation, let $\const{} \coloneqq \const{V}$. We define $f$ as the $\Asp$-quasiconvex envelope $\QC g_\alpha$ of $g_\alpha : \sou \to \mathbb{R}$ defined as follows: given $\alpha \in (0, 1)$, consider $g_\alpha(\var) \coloneqq \const{} \, |P_{V^{\perp}}\var| + \kappa(\var)$, where
	\begin{align*}
		\kappa(\var) \coloneqq \frac{\alpha}{1-\alpha} (|P_V \pol| - |\var - P_V \pol|)_+.
	\end{align*}
	The function $g_{\alpha}$ is non-negative, so that $\QC g_{\alpha} \ge 0$. The first term $\const{} \, |P_{V^{\perp}}\var|$ is positively homogeneous and convex as a function of $\var$. The perturbation $\kappa$ is compactly supported and reaches its maximum $\|\kappa\|_{\infty} = \frac{\alpha}{1-\alpha}|P_V\pol|$ at $\var = P_V \pol$. Moreover, $\kappa(0) = 0$, so that $0 \le \QC g_{\alpha}(0) \le g_{\alpha}(0) = 0$.
	Before going further, let us check that $\QC g_{\alpha}$ is an admissible test function, i.e.\ that $\QC g_{\alpha} \in \predualAqc(\sou)$, and that its recession function is \emph{exactly} $\const{}|P_{V^{\perp}}(\,\cdot\,)|$. Both are elementary properties of the envelope. On the one hand $g_{\alpha}$ is globally Lipschitz, with constant $L_\alpha \coloneqq \const{} + \frac{\alpha}{1-\alpha}$; each competitor $\var \mapsto \int_{[0,1]^d} g_{\alpha}(\var + u(y))\,dy$ is then $L_\alpha$-Lipschitz as well, and since the infimum defining $\QC g_{\alpha}$ is bounded from below by $0$, the envelope is itself $L_\alpha$-Lipschitz, in particular continuous. On the other hand $g_{\alpha} \ge \const{}|P_{V^{\perp}}(\,\cdot\,)|$, and the latter function is convex, hence $\Asp$-quasiconvex; by monotonicity of the envelope $\QC g_{\alpha} \ge \const{}|P_{V^{\perp}}(\,\cdot\,)|$, while trivially $\QC g_{\alpha} \le g_{\alpha} \le \const{}|P_{V^{\perp}}(\,\cdot\,)| + \|\kappa\|_{\infty}$. Altogether,
	\begin{align}\label{spreading:proof:sandwich}
		0 \ \le\ \QC g_{\alpha} - \const{}\,|P_{V^{\perp}}(\,\cdot\,)| \ \le\ \|\kappa\|_{\infty} ,
	\end{align}
	so that $\QC g_{\alpha}$ is the sum of a continuous positively $1$-homogeneous function and a bounded one. For such a function the transform $\widehat{\QC g_{\alpha}}$ extends continuously to $\overline{\ball}_{\sou}$, because $(1-|\var|)\,\big|\QC g_{\alpha} - \const{}|P_{V^{\perp}}(\,\cdot\,)|\big|\big(\tfrac{\var}{1-|\var|}\big) \le (1-|\var|)\|\kappa\|_{\infty} \to 0$ as $|\var| \to 1$; hence $\QC g_{\alpha} \in \predualAqc(\sou)$ and
	\begin{align}\label{spreading:proof:recession}
		(\QC g_{\alpha})^{\infty} = \const{} \, |P_{V^{\perp}}(\,\cdot\,)| .
	\end{align}

	Using \eqref{spreading:proof:recession} and \eqref{spreading:statement:jensen} with $f = \QC g_{\alpha}$, we get that
	\begin{align}\label{spreading:proof:ubound}
		\QC g_{\alpha}(\pol)
		\le \QC g_{\alpha}(0) + \int_{\sphere_{\sou}} (\QC g_{\alpha})^{\infty}(\var) \, d\nu^{\infty}(\var)
		\le \const{} \int_{\sphere_{\sou}} |P_{V^{\perp}}(\var)| \, d\nu^{\infty}(\var).
	\end{align}
	On the other hand, for any periodic $\Asp$-free $u \in \regC^\infty(\mathbb{T}^d,\sou)$ with 0 average, there holds
	\begin{align*}
		\int_{\mathbb{T}^d} g_{\alpha}(\pol + u) \, dx
		=
		\int_{\mathbb{T}^d} \left( \const{} \, |P_{V^{\perp}} (\pol + u)| + \frac{\alpha}{1-\alpha} ( |P_V \pol| - |u + P_{V^\perp} \pol| )_+ \right) dx.
	\end{align*}
	By \Cref{res:transfer}, the first term is bounded from below as
	\begin{align*}
		\int_{\mathbb{T}^d} \const{} \, |P_{V^{\perp}} (\pol + u)| \, dx
		& \geqslant \int_{\mathbb{T}^d} \const{} \, |P_{V^{\perp}} u| \, dx - \const{} |P_{V^{\perp}} \pol| \\
		& \geqslant \|u\|_{L^{1,\infty}(\mathbb{T}^d)} - \const{} |P_{V^{\perp}} \pol|.
	\end{align*}
	To bound the second term, let $E \coloneqq \left\{ |u| \le \alpha (|P_V \pol| - |P_{V^\perp} \pol|) \right\}$. On the set $E$, the integrand satisfies $|P_V \pol| - |u + P_{V^\perp} \pol| \ge |P_V \pol| - |P_{V^\perp} \pol| - |u| \ge (1-\alpha)(|P_V \pol| - |P_{V^\perp} \pol|)$.
	By \eqref{spreading:proof:nondeg} the threshold $\alpha(|P_V \pol| - |P_{V^\perp}\pol|)$ is strictly positive, so Chebyshev's inequality may be applied at that level and gives $\mathscr L^d_{\mathbb{T}^d}(E^c) \le \|u\|_{L^{1,\infty}(\mathbb T^d)} \big/ \big(\alpha(|P_V\pol| - |P_{V^\perp}\pol|)\big)$. Therefore there holds
	\begin{align*}
		\frac{\alpha}{1-\alpha} \int_{\mathbb{T}^d} \big( |P_V \pol| - |u + P_{V^\perp}\pol| \big)_+ \, dx
		&\ \ge\ \frac{\alpha}{1-\alpha} \int_{E} (1 - \alpha)\big(|P_V \pol| - |P_{V^\perp} \pol|\big) \, dx \\
		&\ =\ \alpha \big(|P_V \pol| - |P_{V^\perp} \pol|\big) \left[1 - \mathscr L^d_{\mathbb{T}^d}( E^c )\right] \\
		&\ \ge\ \alpha \big(|P_V \pol| - |P_{V^\perp} \pol|\big) - \|u\|_{L^{1,\infty}(\mathbb{T}^d)}.
	\end{align*}
	Both inequalities, combined with \eqref{spreading:proof:ubound}, imply that
	\begin{align*}
		\const{} \int_{\sphere_{\sou}} |P_{V^{\perp}}(\var)| \, d\nu^{\infty}
		\geqslant \QC g_{\alpha}(\pol)
		= \inf_{u} \int_{\mathbb{T}^d} g_{\alpha}(\pol + u) \, dx
		\geqslant \alpha |P_V \pol| - (\alpha + \const{}) |P_{V^\perp} \pol|.
	\end{align*}
	If $\pol \in V$, taking $\alpha = 1/2$ and $V = \ell_{\pol}$ yields a function $f = \QC g_{1/2}$ satisfying $f(\pol) \ge \frac 12 |\pol| > 0 = f(0)+f^\infty(\pol)$.
	Taking $\pol$ arbitrary in $V \mysm \{0\}$, we conclude to \ref{item:spreading:exif}. That $f$ can be chosen Lipschitz is already contained in the construction: $\QC g_{1/2}$ is globally Lipschitz by the argument leading to \eqref{spreading:proof:sandwich}. On the other hand, \ref{item:spreading:ineq} follows from sending $\alpha \to 1$.
\end{proof}

\begin{remark}\label{rem:homogs}
	The assumption that $\nu^\infty$ on $\sphere_{\sou}$ satisfies~\eqref{spreading:statement:jensen} is equivalent to the fact that for some $\lambda > 0$, or for any $\lambda > 0$, there holds
	\[
	f(\lambda \bar{\nu}^\infty) \le f(0) + \lambda \langle f^\infty,\nu^\infty \rangle \quad \text{for all } f \in \predualAqc(\sou).
	\]
	This follows directly by observing the function $f_\lambda(\var) \coloneqq f(\lambda \var)$ is $\Asp$-quasiconvex if and only if $f$ is $\Asp$-quasiconvex, while also $f_\lambda(0) = f(0)$ and $f_\lambda^\infty = \lambda f^\infty$.
\end{remark}
\begin{remark}[A bound on the barycenter]\label{rem:boundPol}
	By choosing $V = \lin{\pol}$ in \Cref{res:spreading}, we recover the 1-dimensional inequality $|\pol| \le \const{\lin{\pol}} \int_{\sphere_{\sou}} \dist(\var, \lin{\pol}) \, d \nu^{\infty}(\var)$. Using that $ab \leqslant (a^2 + b^2)/2$,
	\begin{align*}
		|\pol|^2
		= \int \left<\pol, \proj_{\lin{\pol}}(X)\right> d\nu^{\infty}(X)
		\leqslant \frac{|\pol|^2}{2} + \frac{1}{2} \int \left(1 - \dist(X,\lin{\pol})^2\right) d\nu^{\infty}(X).
	\end{align*}
	Jointly with the estimate $(|\pol|/\const{\lin{\pol}})^2 \leqslant \int \dist(X,\lin{\pol})^2 \, d\nu^{\infty}$, this yields
	\begin{align*}
		|\pol|^2 \leqslant \frac{|\pol|^2}{2} + \frac{1}{2} \left(1 - (|\pol|/\const{\lin{\pol}})^2\right),
	\end{align*}
	from which we deduce the bound $|\pol| \leqslant \left(1 + (\const{\lin{\pol}})^{-2}\right)^{-1/2}$.
\end{remark}
\begin{remark}[Variance form of the barycenter bound]\label{rem:variance}
	The estimate of \Cref{rem:boundPol} is in fact a variance estimate, and literally so. A probability measure carried by the unit sphere satisfies $\int_{\sphere_{\sou}}|\var|^2 \, d\nu^{\infty} = 1$, so that its variance about the barycenter is
	\[
	\Var(\nu^{\infty}) \coloneqq \int_{\sphere_{\sou}} |\var - \pol|^2 \, d\nu^{\infty}(\var) = 1 - |\pol|^2 .
	\]
	The bound of \Cref{rem:boundPol} is therefore \emph{equivalent} to
	\[
	\Var(\nu^{\infty}) \ \geqslant\ \frac{1}{1 + (\const{\lin{\pol}})^{2}} \ >\ 0 ,
	\]
	and carries no more and no less information: on a sphere, controlling the length of the barycenter and controlling the spread are the same statement. What does say more is to insert the Korn estimate \eqref{eq:korntype:statement}, which converts it into a bound in terms of the position of $\lin{\pol}$ relative to the wave cone: since $\const{\lin{\pol}} \le \ckorn\, d(\lin{\pol},\Lambda_\Asp)^{-(\lfloor d/2\rfloor+2)}$ and $d(\lin{\pol},\Lambda_\Asp) \le 1$,
	\[
	\Var(\nu^{\infty}) \ \geqslant\ \frac{d(\lin{\pol},\Lambda_\Asp)^{2(\lfloor d/2\rfloor+2)}}{d(\lin{\pol},\Lambda_\Asp)^{2(\lfloor d/2\rfloor+2)} + \ckorn^{2}}
	\ \geqslant\ \frac{d(\lin{\pol},\Lambda_\Asp)^{2(\lfloor d/2\rfloor+2)}}{1 + \ckorn^{2}} ,
	\]
	which is \eqref{eq:variance}. In words: the further the barycenter of a concentration profile sits from the wave cone, the more the profile is forced to spread on $\sphere_{\sou}$, at an explicit polynomial rate. The full inequality \eqref{spreading:statement:ineq} of course still says strictly more than this one-dimensional consequence, since it also constrains the component of $\pol$ transversal to $V$ for subspaces $V$ of any dimension.
\end{remark}
%%%%%%%%%%%%%%%%%%%%%%%%%%%%%%%%%%%%%%%%%%%%%%%%%%%%%%%%%%%%%%%%%%%%%%%%%%%%%%
%%  4.  A SIMPLE PROOF OF THE DE PHILIPPIS-RINDLER THEOREM
%%%%%%%%%%%%%%%%%%%%%%%%%%%%%%%%%%%%%%%%%%%%%%%%%%%%%%%%%%%%%%%%%%%%%%%%%%%%%%
\section{A simple proof of the De Philippis--Rindler theorem}\label{sec:dpr}
In this section, we use the construction of an $\Asp$-quasiconvex function in \Cref{res:spreading}\ref{item:spreading:exif} to provide an alternative proof of the result of De Philippis and Rindler. Let us point out that the argument does not rely on the machinery of generalized Young measures, but instead on basic geometric measure theory, standard elliptic regularity, and the classical definition of $\Asp$-quasiconvexity evaluated on smooth periodic fields.
\begin{proof}[Proof of \Cref{thm:DPR}]
	The premise of the proof works by contradiction, similarly to the initial argument in~\cite{dephilippisStructureFreeMeasures2016}. Suppose that the set $A \coloneqq \{ x \in \Omega : \mu/|\mu|(x) \notin \Lambda_{\Asp}\}$ has positive $|\mu|^s$ measure. By Lebesgue's theorem, we can find at least one $|\mu|^s$-Lebesgue point $x_0 \in \Omega \cap A$ of $\mu^s/|\mu^s|$. Without any loss of generality, we may assume $x_0 = 0 \in \Omega$. Since $0$ is a singular density point, there exists a subsequence $(r_j)_j \to 0^+$ such that the blow-up sequence $\mu_j \in \Meas(\ball_1,\sou)$ defined by
	\[
	\mu_{j} \coloneqq \frac{1}{|\mu|^s(\ball_{r_j})} \cdot T^{r_j}_\# \mu,
	\qquad
	T^{r_j}(x) \coloneqq x/r_j,
	\]
	converges weakly-$*$ in $\ball_1$ to a non-trivial $\tau = \pol \lambda$, where $\pol \coloneqq \mu/|\mu|(x_0)$ and $\lambda \in \Meas_+(\ball_1)$ is a positive measure (see \cite[Definition~2.40]{ambrosioFunctionsBoundedVariation2000}).
	Moreover, the fact that $x_0$ is a singular density point and Lebesgue's theorem yield the convergences
	\begin{align}\label{eq:ray}
		\int_{\ball_1}\dist\Big(\tfrac{d\mu_j}{d|\mu_j|},\slin{\pol}\Big)\,d|\mu_j| \; & {\longrightarrow} \; 0 , \\
		|\mu_j|^a(\ball_1) \; & {\longrightarrow} \; 0 . \label{eq:ray2}
	\end{align}
	Since the operator $\Asp$ is homogeneous, we deduce that $(\mu_j)_j$ is a sequence of $\Asp$-free measures and thus, by distributional continuity, $\tau$ is also an $\Asp$-free measure. In particular,
	\[
	\Bsp_{\pol} \lambda \coloneqq \Asp(\pol \lambda) = 0.
	\]
	The fact that $\pol \notin \Lambda_{\Asp}$ is equivalent to the scalar operator $\Bsp_{\pol}$ being elliptic in the sense of H\"ormander. In particular, by a classical regularity result (see Theorem~11.1.1 in \cite{hormanderAnalysisLinearPartial2005}; a generalized Weyl's lemma), we deduce that $\lambda = \theta \, \mathscr L^d$ for some nontrivial $\theta \in \regC^\infty(\R^d;\R^+)$. Alternatively, it is straightforward to prove that $\Bsp_{\pol} u = f$ with $f \in H^m(\Omega)$ implies $u \in H^{m+k}(\Omega)$ via Plancherel's identity (see for instance~\cite[p.~29]{mullerVariationalModelsMicrostructure1999}).
	Because $\theta$ is nontrivial and smooth, we can find a point $y_0 \in \ball_1(0)$ where $\theta_0 \coloneqq \theta(y_0) > 0$. By performing a second blow-up sequence centered at $y_0$, we can extract a new $\Asp$-free sequence $(\tilde \mu_j)_j$ on the unit cube $Q = (-1/2,1/2)^d$ such that $\tilde \mu_j \toweakstar \theta_0 \pol \, \mathscr L^d$ and $|\tilde \mu_j| \toweakstar \theta_0 \mathscr L^d$ as measures on $\overline Q$, with $\int_Q \dist\big(\tfrac{d\tilde\mu_j}{d|\tilde\mu_j|},\slin{\pol}\big) \, d|\tilde\mu_j|^s \to 0$ and $|\tilde \mu_j|^a(Q) \to 0$. In particular, for any Lipschitz $f : \sou \to \R$ with well-defined recession function $f^\infty$ it holds
	\begin{equation}
		\begin{split}\label{eq:various}
			\liminf_{j \to \infty} \left( \int_Q f(\tilde \mu_j^a) \, dx + \int_Q f^\infty(\tilde \mu_j^s/|\tilde \mu_j^s|) \, d|\tilde \mu_j^s| \right)
			& = f(0) + \liminf_{j \to \infty} \int_Q f^\infty(\tilde \mu_j^s/|\tilde \mu_j^s|) \, d|\tilde \mu_j^s| \\
			& = f(0) + f^\infty(\pol)\liminf_{j \to \infty} |\tilde \mu_j^s|(Q) \\
			& = f(0) + \theta_0 \, f^\infty(\pol).
		\end{split}
	\end{equation}
	Fixing $j$ and choosing $0 < \eps_k \ll \delta_k$ with $\delta_k \to 0$ and defining
	\[
	w_{j,k} \coloneqq (T_\#^{1-\delta_k} \tilde \mu_j) \star \rho_{\eps_k}
	\]
	we get
	\[
	\lim_{k \to \infty} \int_Q f(w_{j,k}) \, dx =
	\int_Q f(\tilde \mu_j^a) \, dx + \int_Q f^\infty(\tilde \mu_j^s/|\tilde \mu_j^s|) \, d|\tilde \mu_j^s|.
	\]
	This follows from the fact that, for a fixed measure, the mollified sequence converges to it area-strictly, together with the continuity along area-strictly converging sequences of the functional $\sigma \mapsto \int_Q f(\sigma^a)\,dx + \int_Q f^{\infty}(\sigma^s/|\sigma^s|)\,d|\sigma^s|$ attached to an integrand of linear growth; see \cite{arroyorabasaDimensionalEstimatesRectifiability2019} and, for the underlying Reshetnyak continuity theorem, \cite[Theorem~2.39]{ambrosioFunctionsBoundedVariation2000}. A classical diagonalization argument lets us find a sequence $(v_j)_j \subset \regC^\infty(Q;\sou)$ still satisfying $v_j \mathscr L^d \toweakstar \theta_0 \pol \, \mathscr L^d$,
	\begin{equation}
		\begin{split}\label{eq:various2}
			\liminf_{j \to \infty} \int_Q f(v_{j}) \, dx & =
			\liminf_{j \to \infty} \left( \int_Q f(\tilde \mu_j^a) \, dx + \int_Q f^\infty(\tilde \mu_j^s/|\tilde \mu_j^s|) \, d|\tilde \mu_j^s| \right) \\
			& = f(0) + \theta_0 \, f^\infty(\pol),
		\end{split}
	\end{equation}
	and, crucially, the asymptotic constraint
	\begin{equation}\label{eq:almostAfree}
		\Asp v_j \longrightarrow 0 \qquad \text{strongly in } W^{-k,q}(Q;\tar) \quad \text{for every } 1 < q < \tfrac{d}{d-1} .
	\end{equation}
	Indeed, each $w_{j,k}$ is the mollification of an $\Asp$-free measure, so that $\Asp w_{j,k} = 0$ exactly; the parameters $\delta_k,\eps_k$ and the diagonal extraction may therefore be chosen so that \eqref{eq:various2} and \eqref{eq:almostAfree} hold simultaneously.
	Exploiting \eqref{eq:almostAfree}, the constant-rank property and the Fonseca--M\"uller projection and cut-off techniques, in the form established by Arroyo-Rabasa, De Philippis and Rindler for operators of arbitrary order (see \cite{arroyo-rabasaLowerSemicontinuityRelaxation2020,arroyorabasaDimensionalEstimatesRectifiability2019}; the first-order case is \cite[Lemma 2.15]{fonsecaAQuasiconvexityLowerSemicontinuity1999}, see also \cite{raitaPotentialsAquasiconvexity2019}), we may replace $(v_j)_j$ by a sequence of $\Asp$-free, $Q$-periodic smooth fields $w_j \in \regC^\infty(\mathbb{T}^d;\sou)$ with zero average satisfying (identifying integration on the torus with integration over $Q$)
	\[
	\|v_j - w_j - \theta_0 \pol\|_{L^q(Q)} \to 0 \quad \text{for all $1 < q < \frac{d}{d-1}$}.
	\]
	This, the fact that $f$ is Lipschitz and~\eqref{eq:various2} imply that
	\[
	\lim_{j \to \infty} \int_Q f(\theta_0 \pol + w_{j}) \, dx = f(0) + f^\infty(\theta_0 \pol).
	\]
	If moreover $f$ is $\Asp$-quasiconvex, applying the very definition of $\Asp$-quasiconvexity directly to each $w_j$ we get from the limit above that
	\begin{align}\label{eq:contra}
		f(\theta_0 \pol) \le f(0) + f^{\infty}(\theta_0 \pol).
	\end{align}
	Let $f$ be the Lipschitz $\Asp$-quasiconvex function constructed in \Cref{res:spreading}\ref{item:spreading:exif} with $\theta_0\pol \notin \Lambda_\Asp$. By definition, it satisfies $f(0) + f^{\infty}(\theta_0\pol) < f(\theta_0 \pol)$, which is a contradiction with~\eqref{eq:contra}. This finishes the proof.
\end{proof}

\begin{remark}[Non-homogeneous equation]
	As pointed in \cite[Remark~1.3]{dephilippisStructureFreeMeasures2016}, the conclusion of \Cref{thm:DPR} extends to vector-valued measures $\mu \in \Meas(\Omega,\sou)$ satisfying
	\begin{align*}
		\Asp \mu = \sigma \qquad \text{for some measure } \sigma \in \Meas(\Omega,\tar).
	\end{align*}
	One defines an auxiliary operator $(\mu,\sigma) \mapsto \Asp \mu - \sigma$, and adapts \Cref{res:korntype} for mixed-order operators, the wave cone being that of the principal symbol (terms of highest order) that should have constant rank. The crucial part of the argument is that only the principal symbol of $\mathcal A$ dictates the asymptotic behavior of the blow-up sequence and the relevant $\Asp$-quasiconvexity bounds.
\end{remark}
%%%%%%%%%%%%%%%%%%%%%%%%%%%%%%%%%%%%%%%%%%%%%%%%%%%%%%%%%%%%%%%%%%%%%%%%%%%%%%
%%  5.  PROOF OF THE MAIN RESULTS
%%%%%%%%%%%%%%%%%%%%%%%%%%%%%%%%%%%%%%%%%%%%%%%%%%%%%%%%%%%%%%%%%%%%%%%%%%%%%%
\section{Proof of the main results}\label{sec:proofs}
\begin{proof}[Proof of \Cref{thm:spreading}]
	Let $\boldsymbol{\nu} = (\nu, \lambda, \nu^\infty)$ be a homogeneous $\Asp$-free generalized Young measure with $\lambda \in \Meas_+(\overline{\Omega})$ nontrivial. If $\bar\nu^\infty = 0$ there is nothing to prove, so we assume $\bar\nu^\infty \ne 0$.
	Suppose first that $\lambda^a$ does not vanish identically. By~\eqref{charGAFYM:jensen}, for $\mathscr L^d$-almost every $x$ with $\lambda^a(x) > 0$,
	\begin{align*}
		f(\bar{\nu} + \lambda^a(x) \bar{\nu}^{\infty}) \leqslant \langle f, \nu \rangle + \lambda^a(x) \dpr{f^{\infty}, \nu^{\infty}}
		\qquad \forall f \in \predualAqc(\sou).
	\end{align*}
	Applying the Kristensen--\Raita decoupling (\Cref{res:KR}), this implies
	\begin{align}\label{eq:decoupled_jensen}
		f(\bar{\nu}^{\infty}) \leqslant f(0) + \dpr{f^{\infty}, \nu^{\infty}}
		\qquad \forall f \in \predualAqc(\sou).
	\end{align}
	Suppose now instead that $\lambda^a \equiv 0$, so that $\lambda = \lambda^s$. The barycenter measure $\bar{\boldsymbol\nu}$ is $\Asp$-free by~\eqref{charGAFYM:bary} and its singular part equals $\bar\nu^\infty \lambda^s$; hence \Cref{thm:DPR}, proved in \Cref{sec:dpr}, gives $\bar \nu^\infty \in \Lambda_\Asp$. Since $\mathrm{spt}(\nu^\infty) \subset \sou_\Asp$ by~\eqref{eq:support}, since every $f \in \predualAqc(\sou)$ is $\Lambda_\Asp$-convex \cite[Proposition~3.4]{fonsecaAQuasiconvexityLowerSemicontinuity1999}, and since $f^\infty$ admits a nonempty convex subdifferential at each point of $\Lambda_\Asp$ \cite[Theorem~1.1]{kirchheimRankOneConvex2016}, Jensen's inequality gives $\dpr{f^\infty,\nu^\infty} \ge f^\infty(\bar \nu^\infty)$, while the monotonicity of the difference quotients of $f$ along the direction $\bar \nu^\infty \in \Lambda_\Asp$ yields
	\[
	f(\bar \nu^\infty) - f(0) \le \lim_{t \to \infty} \frac{f(t\bar \nu^\infty) - f(0)}{t} = f^\infty(\bar \nu^\infty).
	\]
	Combining the two displays gives~\eqref{eq:decoupled_jensen} in this case as well.
	In either case, the hypothesis of \Cref{res:spreading}\ref{item:spreading:ineq} is satisfied, so that for $\pol = \bar{\nu}^{\infty}$, there holds
	\begin{align*}
		|P_V \pol| - (1+\const{V}) |P_{V^\perp} \pol| \leqslant \const{V} \int_{\sphere_{\sou}} \dist(\var, V) \, d\nu^{\infty}(\var). & \qedhere
	\end{align*}
\end{proof}

To prove \Cref{thm:main}, we first show the following intermediate result.
\begin{lemma}\label{lem:intermediate}
	Let $\Asp$ satisfy the constant rank property and let $V \subset \sou$ be an elliptic subspace, i.e.\ $V \cap \Lambda_{\Asp} = \{0\}$.
	Suppose that $0 \le \alpha, \beta < \pi/2$ are angles satisfying
	\begin{equation}\label{eq:alpha_bound}
		\alpha < \frac{\cos(\beta)}{1+2\ckorn} \, d(V,\Lambda_{\Asp})^{\lfloor d/2 \rfloor + 2},
	\end{equation}
	where $\ckorn = \ckorn(\Asp,d) \ge 1$ is the constant appearing in the Korn estimate~\eqref{eq:korntype:statement} of \Cref{res:korntype}. If $\mathcal K \subset \sou$ is a closed convex cone which is elliptic, $\mathcal K \cap \Lambda_{\Asp} = \{0\}$, and which satisfies, for some $\dircone \in V \cap \sphere_{\sou}$, the aperture constraints
	\begin{align}\label{def:cone_constraints2}
		\left\{\begin{aligned}
			|P_V X| &\ge \cos(\alpha)|X| \\
			\dpr{X,Z} &\ge \cos(\beta) |X|
		\end{aligned}
		\right. \quad \qquad \forall\, X \in \mathcal K,
	\end{align}
	then every uniformly bounded sequence $(u_j)_j \subset L^1(\Omega;\sou)$ satisfying
	\begin{equation}\label{interlem:hyp}
		\begin{cases}
			\begin{aligned}
				\Asp u_j &= 0 \quad \text{in the sense of distributions,}\\
				\quad \dist(u_j, \mathcal K) &\to 0 \quad \text{in $L^1_{\mathrm{loc}}(\Omega)$,}
			\end{aligned}
		\end{cases}
	\end{equation}
	is equi-integrable on any compact subset $\bdsubset \subset \Omega$.
\end{lemma}
\begin{proof}
	Assume by contradiction that $(u_j)_j$ fails to be equi-integrable on some compact $\bdsubset \subset \Omega$.
	Up to taking a subsequence, we may assume that $(u_j |_{\bdsubset})_j$ still fails to be equi-integrable, and generates an $\Asp$-free Young measure $\boldsymbol{\nu} = (\nu, \lambda, \nu^{\infty})$ on $\bdsubset$. Since $(u_j)_j$ is not equi-integrable, there holds $\lambda(\bdsubset) > 0$ \cite[Theorem~2.9]{alibertNonUniformIntegrabilityGeneralized1997}. By~\eqref{interlem:hyp} and the definition of the convergence in the sense of Young measures,
	\begin{equation}\label{main:proof:conc}
		\int_{\bdsubset} \int_{\sou} \dist(\var, \mathcal K) \, d\nu_x(\var) \, dx + \int_{\bdsubset} \int_{\sphere_{\sou}} \dist(\var, \mathcal K) \, d\nu_x^{\infty}(\var) \, d\lambda = 0.
	\end{equation}
	Hence $\nu_x^{\infty}$ is concentrated on $\mathcal K$ for $\lambda$-almost any point. By the convexity of $\mathcal K$, the averaged probability measure $\mu^{\infty} \coloneqq \lambda(\bdsubset)^{-1} \int_{\bdsubset} \nu_x^{\infty} \, d\lambda$ is still so, and its barycenter $\pol \coloneqq \bar \mu^{\infty}$ belongs to $\mathcal K$. Moreover $\dpr{\pol,\dircone} = \int \dpr{\var,\dircone} \, d\mu^\infty(\var) \ge \cos(\beta) > 0$, so that $\pol \ne 0$; since $\mathcal K \cap \Lambda_{\Asp} = \{0\}$, we conclude $\pol \notin \Lambda_{\Asp}$.
	We check that $\mu^{\infty}$ satisfies the hypothesis of \Cref{res:spreading}\ref{item:spreading:ineq}; note first that $\mathrm{spt}\,\mu^\infty \subset \sou_\Asp$ by~\eqref{eq:support}.
	Combining the Jensen inequalities \eqref{charGAFYM:jensen} with the Kristensen--\Raita decoupling (\Cref{res:KR}), there holds for $\mathscr L^d$-almost every $x \in \bdsubset$ with $\lambda^a(x)>0$ that
	\begin{equation}\label{main:proof:anotherJensen}
		f(\var + \bar{\nu}_x^{\infty}) \le f(\var) + \dpr{f^{\infty}, \nu_x^{\infty}}
	\end{equation}
	for any $f \in \predualAqc(\sou)$ and $\var \in \sou$.
	On the other hand, \Cref{thm:DPR} implies that for any $\Asp$-free Young measure $\boldsymbol{\nu} = (\nu,\lambda,\nu^{\infty})$, the barycenter $\bar\nu^{\infty}_x$ belongs to $\Lambda_{\Asp}$ for $\lambda^s$-almost any $x \in \Omega$. Now, any $f \in \predualAqc(\sou)$ is convex along the directions of the wave cone $\Lambda_{\Asp}$ \cite[Proposition~3.4]{fonsecaAQuasiconvexityLowerSemicontinuity1999}, and $f^{\infty}$ admits a nonempty convex subdifferential at any point of $\Lambda_{\Asp}$ \cite[Theorem~1.1]{kirchheimRankOneConvex2016}. It follows that for any $\var \in \sou$, there holds
	\begin{equation}\label{main:proof:singpts}
		f(\var + \bar \nu^{\infty}_x) - f(\var)
		\le f^{\infty}(\bar \nu^{\infty}_x)
		\le \dpr{f^{\infty}, \nu_x^{\infty}},
	\end{equation}
	this time for $\lambda^s$-almost every point $x \in \bdsubset$.
	Let $f \in \predualAqc(\sou)$. Applying \eqref{main:proof:anotherJensen}-\eqref{main:proof:singpts} to $\var = \pol - \bar{\nu}_x^{\infty}$ and integrating against the normalized measure $\lambda/\lambda(\bdsubset)$, we get
	\begin{equation*}
		f(\pol) \le \lambda(\bdsubset)^{-1} \int_{\bdsubset} f(\pol -\bar{\nu}_x^{\infty}) d\lambda(x) +
		\dpr{f^{\infty}, \mu^{\infty}}.
	\end{equation*}
	Note that $\mu \coloneqq \lambda(\bdsubset)^{-1} (\pol -\bar{\nu}_x^{\infty})_{\#} \lambda$ satisfies $\mathrm{spt}(\mu) \subset \sou_\Asp$ and has null barycenter, since
	\begin{equation*}
		\int_{\sou} \var \, d\mu(\var)
		= \lambda(\bdsubset)^{-1} \int_{\bdsubset} (\pol - \bar{\nu}_x^{\infty}) d\lambda(x)
		= \pol - \int_{\sou} \var \, d\mu^{\infty}(\var)
		= 0.
	\end{equation*}
	The displayed inequality is therefore exactly~\eqref{KR:hyp} for the triple $(\mu,1,\mu^{\infty})$. Applying the Kristensen--\Raita decoupling once more, this time to $(\mu,1,\mu^{\infty})$, we deduce that $f(\bar \mu^{\infty}) = f(\pol) \le f(0) + \dpr{f^{\infty}, \mu^{\infty}}$ for all $f \in \predualAqc(\sou)$. From \Cref{res:spreading}\ref{item:spreading:ineq}, we obtain the spreading bound:
	\begin{equation}\label{eq:spread_eval}
		|P_V \pol| - (1+\const{V}) |P_{V^\perp} \pol| \le \const{V} \int_{\sphere_{\sou}} |P_{V^{\perp}}(\var)| \, d\mu^{\infty}(\var).
	\end{equation}
	We now evaluate this inequality using the geometric constraints~\eqref{def:cone_constraints2}. Because $\mu^\infty$ is a probability measure supported on $\mathcal K \cap \sphere_{\sou}$, the integral of the orthogonal projection is bounded by $\sin(\alpha)$. By Jensen's inequality, the barycenter also satisfies $|P_{V^\perp} \pol| \le \sin(\alpha)$. Furthermore, because the unit vector $\dircone$ belongs to $V$, the Cauchy--Schwarz inequality and the fact that $\mu^\infty$ is supported on $S_\sou$ bounds the projection onto $V$ from below:
	\begin{equation*}
		|P_V \pol|
		\ge \dpr{P_V \pol, \dircone}
		= \dpr{\pol, \dircone}
		= \int_{\sphere_{\sou}} \dpr{\var, \dircone} \, d\mu^\infty(\var)
		\ge \cos(\beta).
	\end{equation*}
	Substituting these bounds into \eqref{eq:spread_eval} yields $\cos(\beta) - (1+\const{V}) \sin(\alpha) \le \const{V} \sin(\alpha)$, which implies $\cos(\beta) \le (1 + 2\const{V}) \sin(\alpha)$. Using the standard inequality $\sin(\alpha) \le \alpha$, this forces $\alpha \ge \frac{\cos(\beta)}{1+2\const{V}}$. By \Cref{res:korntype}, we have $\const{V} \le \ckorn \, d(V,\Lambda_\Asp)^{-(\lfloor d/2 \rfloor + 2)}$. Since $d(V,\Lambda_\Asp) \le 1$, we can coarsely bound
	\begin{equation*}
		\alpha \ge \frac{\cos(\beta)}{1+2\ckorn} \, d(V,\Lambda_{\Asp})^{\lfloor d/2 \rfloor + 2},
	\end{equation*}
	which contradicts the choice of $\alpha$ in \eqref{eq:alpha_bound}. Hence $(u_j)_j$ must be equi-integrable on $\bdsubset$.
\end{proof}
\begin{remark}[Symmetric-aperture cones]\label{rem:equal}
	We observe that in the symmetric case where $\alpha = \beta$, we can exploit the exact trigonometric identity to drop the cosine term. The inequality $\cos(\alpha) \le (1 + 2\const{V}) \sin(\alpha)$ rearranges to $\cot(\alpha) \le 1 + 2\const{V}$, or equivalently, $\alpha \ge \arctan(1/(1+2\const{V}))$. Since $t = 1/(1+2\const{V})$ is in $(0, 1]$, the strict concavity of the arctangent function enforces the lower bound $\arctan(t) \ge \frac{\pi}{4} t$. We thus obtain $\alpha \ge \frac{\pi}{4(1+2\const{V})}$. Bounding $\const{V}$ as before, we conclude that a non-trivial measure requires $\alpha \ge \frac{\pi}{4(1+2\ckorn)} d(V, \Lambda_{\Asp})^{\lfloor d/2 \rfloor + 2}$, recovering a scaling without independent trigonometric parameters.
\end{remark}
We deduce the following corollary on vector-valued measures in the kernel of $\Asp$.
\begin{corollary}\label{res:main:measure}
	Let $V,\alpha,\beta$ and $\mathcal K$ be as in \Cref{lem:intermediate}.
	Let $(\mu_j)_{j \in \mathbb{N}} \subset \Meas(\Omega;\sou)$ be a sequence of vector-valued measures such that $|\mu_j|(\Omega) \leqslant C$ for some constant $C$, while
	\begin{align*}
		\Asp \mu_j = 0 \quad \text{ in the sense of distributions on } \Omega,
	\end{align*}
	and there holds
	\begin{align*}
		\int_{\Omega} \dist\left(\frac{d\mu_j}{d|\mu_j|}, \mathcal K\right) d|\mu_j| \underset{j \to \infty}{\longrightarrow} 0.
	\end{align*}
	Then, for any compact $K \subset \Omega$, $(|\mu_j|^a)_j$ is equi-integrable on $K$ and $|\mu_j|^s(K) \to 0$.
\end{corollary}
\begin{proof}
	Let $K\subset \Omega$ be compact, and $(A_j)_j$ be a sequence of subsets of $K$ such that $\mathscr L^d(A_j) \to 0$ when $j \to \infty$.
	We claim that
	\begin{align}\label{main:measure:limit}
		\lim_{j \to \infty} |\mu_j|(A_j) = 0.
	\end{align}
	Let $r > 0$ be such that the open set $K^r \coloneqq \left\{ x \st \dist(x,K) < r \right\}$ is contained in $\Omega$. For each $j$, let $D_j \subset A_j$ be a compact set such that $|\mu_j|(D_j) \geqslant |\mu_j|(A_j)/2$. By compactness of $D_j$, we can construct a function $\varphi_j \in \regC_c(K^{r/4};[0,1])$ such that $\varphi_j(x) = 1$ if $x \in D_j$ and $\mathscr L^d(\{\varphi_j > 0\}) \leqslant \mathscr L^d(D_j) + 2^{-j}$.
	We claim that for any $\var \in \sou$, there holds
	\begin{equation}\label{eq:XY}
		|P_{V}\var| \geqslant \cos(\alpha) |\var| - 2 \dist(\var,\mathcal K).
	\end{equation}
	Indeed, for $\varalt \in \mathcal K$ such that $|\var - \varalt| = \dist(\var,\mathcal K)$, there holds
	\begin{align*}
		|P_{V}\varalt| \leqslant |P_{V}\var| + |P_{V}(\varalt-\var)| \leqslant |P_{V}\var| + \dist(\var,\mathcal K),
	\end{align*}
	while due to the aperture constraints on $\mathcal K$ and positivity of $\cos(\alpha)$,
	\begin{align*}
		|P_{V}\varalt| \geqslant \cos(\alpha) |\varalt| \geqslant \cos(\alpha) |\var| - \cos(\alpha) |\var - \varalt| \geqslant \cos(\alpha) |\var| - \dist(\var,\mathcal K).
	\end{align*}
	Gathering these two inequalities yields the desired inequality.
	Casting the pointwise value $\var = \frac{d\mu_j}{d|\mu_j|}(x)$ into~\eqref{eq:XY} and integrating over $|\mu_j|$ yields:
	\begin{align*}
		\int \varphi_j \left|P_V\left(\frac{d\mu_j}{d|\mu_j|}\right)\right| \, d|\mu_j|
		\geqslant \cos(\alpha) |\mu_j|(D_j) - 2 \int_{\Omega} \dist\left(\frac{d\mu_j}{d|\mu_j|}, \mathcal K\right) \, d|\mu_j|.
	\end{align*}
	Let $(\rho_\eps)_{\eps}$ be a family of standard mollifiers. For a \emph{fixed} $j$, the mollifications $\mu_j * \rho_\eps$ converge strictly to $\mu_j$ as $\eps \to 0$, so that $|P_V(\mu_j*\rho_\eps)| \, \mathscr L^d \toweakstar |P_V \mu_j|$ (see \cite[Prop.~1.62 and Thm.~2.2]{ambrosioFunctionsBoundedVariation2000}); since $|P_V (\cdot)|$ is a continuous seminorm and $\varphi_j$ is continuous with compact support, we may therefore find $0 < \eps_j < r/4$ small enough so that the smoothed density $u_j \coloneqq \mu_j * \rho_{\eps_j}$ satisfies
	\begin{align*}
		\int \varphi_j |P_V u_j| \, dx
		\geqslant \int \varphi_j \left|P_V\left(\frac{d\mu_j}{d|\mu_j|}\right)\right| d|\mu_j| - 2^{-j}.
	\end{align*}

	Since $\eps_j < r/4$, the density $u_j$ is well defined and $\Asp$-free on the open set $K^{r/2}$, where it satisfies $\|u_j\|_{L^1(K^{r/2})} \leqslant |\mu_j|(\Omega) \leqslant C$. Since $K^{r} \subset \Omega$ and the function $\dist(\cdot,\mathcal K)$ is convex and positively $1$-homogeneous, Jensen's inequality gives
	\begin{align*}
		\int_{\overline{K^{r/2}}} \dist(u_j, \mathcal K) \, dx
		\leqslant \int_{K^r} \dist\left(\frac{d\mu_j}{d|\mu_j|}, \mathcal K\right) \, d|\mu_j|
		\underset{j \to \infty}{\longrightarrow} 0.
	\end{align*}

	Applying \Cref{lem:intermediate} on the open set $K^{r/2}$ and the compact subset $\overline{K^{r/4}} \subset K^{r/2}$, which contains $\{\varphi_j > 0\}$ for every $j$, we discover that the sequence $(u_j)_j$ must be equi-integrable there.

Therefore, gathering the previous inequalities and recalling that $\mathscr L^d(\{\varphi_j > 0\}) \to 0$,
	\begin{align*}
		\limsup_{j \to \infty} \cos(\alpha) |\mu_j|(D_j)
		& \leqslant \limsup_{j \to \infty} \int_{\{\varphi_j > 0\}} |P_V u_j(x)| dx
		\\& \leqslant \limsup_{j \to \infty} \int_{\{\varphi_j > 0\}} |u_j(x)| dx = 0.
	\end{align*}
	This implies \eqref{main:measure:limit}.
	We deduce that $|\mu_j|^s \to 0$ on each bounded subset of $\Omega$; otherwise, there would exist equibounded compacts $K_{j} \subset \Omega$ of null Lebesgue measure such that $0 < \limsup_j |\mu_j|^s(K_j) \leqslant \limsup_j |\mu_j|(K_j)$, contradicting \eqref{main:measure:limit}. Similarly, $(|\mu_j|^a)_j$ must be equi-integrable on compact subsets of $\Omega$.
	Otherwise, one could find some bounded $U \Subset \Omega$ and $\eps > 0$ such that for any $\delta > 0$, there exist $j$ and $A \subset U$ of measure less than $\delta$ for which $\int_{A} |\mu_j|^a dx \geqslant \eps$. For each $k$, pick a measurable $A_k \subset U$ of measure $|A_k| \leqslant 2^{-k}$, and $j_k$ such that $|\mu_{j_k}|^a(A_k) \geqslant \eps$. The equibounded sequence $(A_k)_k$ has vanishing measure, but
	\begin{align*}
		\limsup_{k \to \infty} |\mu_{j_k}|(A_k)
		\geqslant \limsup_{k \to \infty} |\mu_{j_k}|^a(A_k)
		\geqslant \eps,
	\end{align*}
	contradicting \eqref{main:measure:limit}.
\end{proof}
In the case where the sequence $(\mu_j)_j$ is additionally assumed to converge weakly-$*$ towards a measure $\mu$, we obtain our main result.
\begin{proof}[Proof of \Cref{thm:main}]
	Let $(\mu_j)_j \subset \Meas(\Omega;\sou)$ be a uniformly bounded sequence of $\Asp$-free measures converging weakly-$*$ towards $\mu \in \Meas(\Omega;\sou)$ and satisfying the third hypothesis of \Cref{thm:main}, that is,
	\begin{align*}
		\int_{\Omega} \dist\left(\frac{d \mu_j}{d|\mu_j|}(x),\mathcal K\right) d|\mu_j|(x) \to 0.
	\end{align*}

	\emph{Step 1: the singular parts vanish, globally.}
	This first step uses neither the weak-$*$ convergence of $(\mu_j)_j$ nor any localization; it rests solely on the fact that $\mathcal K$ and $\Lambda_{\Asp}$ are closed cones meeting only at the origin. If $\Lambda_\Asp = \{0\}$, then \Cref{thm:DPR} forces $|\mu_j|^{s} \equiv 0$ and there is nothing to prove. Otherwise,
	\[
	\delta \coloneqq \min\big\{ \dist(\varalt,\mathcal K) \st \varalt \in \Lambda_\Asp, \ |\varalt| = 1 \big\} > 0 ,
	\]
	while \Cref{thm:DPR} gives $\frac{d\mu_j}{d|\mu_j|} \in \Lambda_\Asp$ at $|\mu_j|^{s}$-almost every point. Consequently
	\[
	\delta \, |\mu_j|^{s}(\Omega) \ \le\ \int_{\Omega} \dist\Big(\tfrac{d\mu_j}{d|\mu_j|},\mathcal K\Big) \, d|\mu_j| \ \longrightarrow\ 0 ,
	\]
	that is, $|\mu_j|^{s} \to 0$ strongly in $\Meas(\Omega)$.

	\emph{Step 2: the absolutely continuous parts converge weakly in $L^1_{\mathrm{loc}}$.}
	Fix a compact set $K \subset \Omega$. By \Cref{res:main:measure}, we deduce that $(|\mu_j|^a)_j$ is equi-integrable on $K$ and that $|\mu_j|^s(K) \to 0$. By the Dunford--Pettis theorem \cite[Theorem~1.38]{ambrosioFunctionsBoundedVariation2000}, $(\mu_j^a)_j$ is therefore weakly relatively compact in $L^1(K;\sou)$. Since $|\mu_j|^s(K) \to 0$, every weak limit point of $(\mu_j^a)_j$ must coincide with $\mu \resmes K$; hence $\mu\resmes K = v \, \mathscr L^d \resmes K$ for some $v \in L^1(K;\sou)$ and the whole sequence $(\mu_j^a)_j$ converges weakly to $v$ in $L^1(K;\sou)$. As $K \subset \Omega$ was an arbitrary compact set, the conclusion $\frac{d\mu_j}{d\mathscr L^d} \rightharpoonup v$ in $L^1_{\mathrm{loc}}(\Omega;\sou)$ follows.

	\emph{Step 3: the limit density is globally integrable.}
	The exhaustion of Step 2 is needed only for the weak convergence. Indeed, $\mu \ll \mathscr L^d$ on all of $\Omega$ by Step 2, and its density $v$ satisfies $\int_{\Omega} |v| \, dx = |\mu|(\Omega) < \infty$, so that $v \in L^1(\Omega;\sou)$. Together with Step 1, this is exactly the assertion of \Cref{thm:main}.
\end{proof}

\begin{remark}[No higher integrability]\label{rem:no_higher_int}
	Let $V$, $\alpha$, $\beta$ and $\mathcal K \in \mathcal C_{\alpha,\beta}(V)$ be as in \Cref{thm:main}. Fix $\zeta \in \Lambda_\Asp \mysm \{0\}$ and $\xi_0 \in \sphere^{d-1}$ with $\Afv(\xi_0)\zeta = 0$, and fix $\dircone \in \mathcal K$ with $|\dircone| = 1$. For $\eps \in (0,\tfrac12)$ and $a>0$, let $h_{\eps,a} : \R \to \R$ be the $1$-periodic function of zero mean equal to $a$ on a set of measure $\eps$ in each period and to $-a\eps/(1-\eps)$ elsewhere (mollified, if smoothness is desired). Since $\Afv(\xi_0)\zeta = 0$, the plane wave
	\[
	w_{\eps,a}(x) \coloneqq \zeta \, h_{\eps,a}(x \cdot \xi_0)
	\]
	is $\Asp$-free, and on a bounded $\Omega$ one has $\|w_{\eps,a}\|_{L^1} \simeq a\eps$ and $\|w_{\eps,a}\|_{L^p}^p \simeq a^p \eps$. Given $p > 1$, choose $s \in (1/p,1)$, put $a_j = \eps_j^{-s}$ with $\eps_j \downarrow 0$, and set
	\[
	v_j \coloneqq \dircone + w_{\eps_j,a_j} .
	\]
	Then $(v_j)_j$ is $\Asp$-free, bounded in $L^1$, and satisfies $\dist(v_j,\mathcal K) \le |w_{\eps_j,a_j}| \to 0$ in $L^1$ because $\dircone \in \mathcal K$; yet $\|v_j\|_{L^p} \gtrsim \eps_j^{1/p - s} \to \infty$. Letting $s = s_j \uparrow 1$ at a suitable rate, the same crests blow up in $L\log L$ while still tending to $\dircone$ in $L^1$. There is no contradiction with \Cref{thm:main}: $w_{\eps_j,a_j} \to 0$ in $L^1$, so $(v_j)_j$ is equi-integrable, as it must be. What the construction shows is that the $L^1$-conclusion of \Cref{thm:main} cannot be upgraded to any higher-integrability scale, the obstruction being the arbitrarily thin and tall crests that the wave cone always makes available inside an $L^1$-neighbourhood of $\mathcal K$.
\end{remark}

%%%%%%%%%%%%%%%%%%%%%%%%%%%%%%%%%%%%%%%%%%%%%%%%%%%%%%%%%%%%%%%%%%%%%%%%%%%%%%
%%  APPENDICES
%%%%%%%%%%%%%%%%%%%%%%%%%%%%%%%%%%%%%%%%%%%%%%%%%%%%%%%%%%%%%%%%%%%%%%%%%%%%%%
\appendix
%%%%%%%%%%%%%%%%%%%%%%%%%%%%%%%%%%%%%%%%%%%%%%%%%%%%%%%%%%%%%%%%%%%%%%%%%%%%%%
%%  A.  BOUNDS ON HIGHER-ORDER DERIVATIVES OF THE PSEUDO-INVERSE
%%%%%%%%%%%%%%%%%%%%%%%%%%%%%%%%%%%%%%%%%%%%%%%%%%%%%%%%%%%%%%%%%%%%%%%%%%%%%%
\section{Bounds on higher-order derivatives of the pseudo-inverse}\label{appx:bound}
In this section, we provide a careful derivation of the bounds on the higher-order derivatives of a general Moore--Penrose pseudo-inverse matrix. This computation is instrumental for establishing the explicit power of the distance scaling in \Cref{thm:main}. To stay self-contained, we provide it here in full generality.
Recall that the Moore--Penrose pseudo-inverse $W^{\dagger}$ of $W$ is the unique matrix satisfying $W W^{\dagger} W = W$, $W^{\dagger} W W^{\dagger} = W^{\dagger}$, $(W W^{\dagger})^{*} = W W^\dagger$, and $(W^{\dagger} W)^* = W^{\dagger} W$.
The interested reader will find details on $W^{\dagger}$ in \cite[Section~B.4.4]{luNumericalMatrixDecomposition2021}. We bound the derivatives of $W^{\dagger}$ by a slight extension of the computations of Golub and Pereyra for the first-order case~\cite{golubDifferentiationPseudoInversesNonlinear1973}. Intuitively, every order of differentiation of the pseudo-inverse brings only one power to the estimate. However, a na\"ive application of the product rule suggests a bound with one power more than expected. To obtain the correct power, the induction argument relies on the identities of the pseudo-inverse that benefit from particular cancellations.
\begin{proposition}\label{res:boundPseudo}
	Let $W : \mathbb{R}^d \mysm \{0\} \to \mathrm{End}(\sou)$ be a smooth $0$-homogeneous map. Assume that the rank of $W(\xi)$ is independent of $\xi$. Then for any $k \in \mathbb{N}_*$, there exists a constant $C_k$ depending only on $\sup_{|\xi| = 1} \sup_{|\alpha| \leqslant k} \|\partial^{\alpha} W(\xi)\|$ such that
	\begin{align*}
		\|\partial^{\alpha} W^{\dagger}(\xi)\| \leqslant C_k |\xi|^{-|\alpha|} \|W^{\dagger}(\xi)\|^{|\alpha|+1}
		\qquad \forall \xi \in \mathbb{R}^d \mysm \{0\}, \qquad \forall \alpha \text{ such that } |\alpha| \leqslant k.
	\end{align*}
\end{proposition}

The rest of this appendix is devoted to the proof of \Cref{res:boundPseudo}, which will follow from \Cref{lem:A2} and \Cref{lem:A3} below.
Since $W$ is 0-homogeneous as a function of $\xi$, so is $W^{\dagger}$, and one has
\begin{align*}
	\|\partial^{\alpha} W^{\dagger}(\xi)\| = |\xi|^{-|\alpha|} \|\partial^{\alpha} W^{\dagger}(\xi/|\xi|)\|
	\qquad \forall \xi \in \mathbb{R}^d \mysm \{0\}.
\end{align*}
Hence we can assume that $|\xi| = 1$. Denote $P \coloneqq W^{\dagger} W$ and $Q \coloneqq W W^{\dagger}$ the orthogonal projections on the image of $W^*$ and $W$ respectively. In the sequel of the proof, the letter $C_k$ stands for a nonnegative constant that may change between expressions, but depends solely on the order of derivation $k \in \mathbb{N}_*$ and $\sup_{|\xi| = 1} \|\partial^{\alpha} W(\xi)\|$ for $|\alpha| \leqslant k$. Define the induction properties
\begin{align}
	&\text{There exists } C_k \text{ s.t. } &
	\max(\|\partial^{\alpha} Q\|,\|\partial^{\alpha} P\|) &\leqslant C_k \|W^{\dagger}\|^{|\alpha|} & \text{ if } |\alpha| \leqslant k. \label{ind:QP}\tag{H1\textsubscript{$k$}} \\
	&\text{There exists } C_k \text{ s.t. } &
	\|\partial^{\alpha} W^{\dagger}\| &\leqslant C_k \|W^{\dagger}\|^{|\alpha|+1} & \text{ if } |\alpha| \leqslant k. \label{ind:W}\tag{H2\textsubscript{$k$}}
\end{align}
\begin{lemma}\label{lem:A2}
	The property \eqref{ind:QP} holds for $k = 1$. Moreover, if \eqref{ind:QP}-\eqref{ind:W} hold up to the order $K$, then \eqref{ind:QP} holds at order $K+1$.
\end{lemma}
\begin{proof}
	Taking derivatives of the identity $W = Q W$, we get
	\begin{align*}
		\partial^{\alpha} W
		= \partial^{\alpha} \left(Q W\right)
		= \sum_{\alpha_0 + \alpha_1 = \alpha} \partial^{\alpha_0} Q \partial^{\alpha_1} W
		= \sum_{\substack{\alpha_0 + \alpha_1 = \alpha \\ |\alpha_0| < |\alpha|}} \partial^{\alpha_0} Q \partial^{\alpha_1} W + \partial^{\alpha} Q W.
	\end{align*}
	Isolating $\partial^{\alpha} Q W$ and multiplying by $W^{\dagger}$, we obtain
	\begin{align}\label{parQQ}
		\partial^{\alpha} Q Q
		= \partial^{\alpha} Q W W^{\dagger}
		= \partial^{\alpha} W W^{\dagger} - \sum_{\substack{\alpha_0 + \alpha_1 = \alpha \\ |\alpha_0| < |\alpha|}} \partial^{\alpha_0} Q \partial^{\alpha_1} W W^{\dagger}.
	\end{align}
	On the other hand, since $Q = Q^{|\alpha|+1}$ for any multi-index $\alpha$, there holds
	\begin{align}\label{parQ}
		\partial^{\alpha} Q
		= \partial^{\alpha} Q^{|\alpha|+1}
		= \sum_{\substack{\sum_{j=0}^{|\alpha|} \alpha_j = \alpha \\ |\alpha_j| < |\alpha|}} \prod_{j=0}^{|\alpha|} \partial^{\alpha_j} Q + \sum_{j=0}^{|\alpha|-1} Q^j \partial^{\alpha} Q Q + Q \partial^{\alpha} Q.
	\end{align}
	The terms $\partial^{\alpha} Q Q$ are given by \eqref{parQQ}. Now, as $Q$ is self-adjoint, there holds $Q \partial^{\alpha} Q = Q^* \partial^{\alpha} Q^* = \left(\partial^{\alpha} Q Q\right)^*$. In the case where $|\alpha| = 1$, we obtained the expression
	\begin{align*}
		\partial^{\alpha} Q
		= \partial^{\alpha} Q Q + Q \partial^{\alpha} Q
		= \partial^{\alpha} W W^{\dagger} - Q \partial^{\alpha} W W^{\dagger} + \left[\partial^{\alpha} W W^{\dagger} - Q \partial^{\alpha} W W^{\dagger}\right]^*.
	\end{align*}
	Therefore $\|\partial^{\alpha} Q\| \leqslant 4 \|\partial^{\alpha} W\| \|W^{\dagger}\|$. As $P = W^{\dagger} W = (W^* (W^*)^{\dagger})^*$, the same computations provide an expression for $\partial^{\alpha} P^*$ with $W^*$ in place of $W$. As $\|\partial^{\alpha} W\| = \|\partial^{\alpha} W^*\|$ and $\|W^{\dagger}\| = \|(W^*)^{\dagger}\|$, the same bound holds, and we conclude that the property \eqref{ind:QP} holds for $k = 1$. Now, assuming that \eqref{ind:QP}-\eqref{ind:W} hold up to the order $K \geqslant 1$, let $\alpha$ be such that $|\alpha| = K+1$. We may use \eqref{parQQ} to write
	\begin{align*}
		\|\partial^{\alpha} Q Q\|
		\leqslant \big(\|\partial^{\alpha} W\| + \sum_{\substack{\alpha_0 + \alpha_1 = \alpha \\ |\alpha_0| < |\alpha|}} \|\partial^{\alpha_0} Q\| \| \partial^{\alpha_1} W\|\big) \|W^{\dagger}\|
		\leqslant C_{K+1} \|W^{\dagger}\|^{|\alpha|}.
	\end{align*}
	Using \eqref{parQ}, this yields
	\begin{align*}
		\|\partial^{\alpha} Q\|
		& \leqslant \sum_{\substack{\sum_{j=0}^{|\alpha|} \alpha_j = \alpha \\ |\alpha_j| < |\alpha|}} \prod_{j=0}^{|\alpha|} (C_{K} \|W^{\dagger}\|^{|\alpha_j|}) + |\alpha| C_{K} \|W^{\dagger}\|^{|\alpha|} \\
		& \qquad + C_{K+1} \|W^{\dagger}\|^{|\alpha|}
		\leqslant C_{K+1} \|W^{\dagger}\|^{|\alpha|}.
	\end{align*}
	The same bound holds for $P$ with $W^*$ in place of $W$, so that \eqref{ind:QP} holds for $k = K+1$.
\end{proof}
We now complete the induction.
\begin{lemma}\label{lem:A3}
	The property \eqref{ind:W} holds for $k = 1$. Moreover, if \eqref{ind:QP}-\eqref{ind:W} hold up to the order $K$, and \eqref{ind:QP} up to $k = K+1$, then \eqref{ind:W} holds at order $K+1$.
\end{lemma}
\begin{proof}
	Taking derivatives of the identity $W^{\dagger} = W^{\dagger} W W^{\dagger}$, there holds
	\begin{align}\label{main}
		\partial^{\alpha} W^{\dagger}
		= \partial^{\alpha} W^{\dagger} Q + P \partial^{\alpha} W^{\dagger} + \sum_{\substack{\alpha_0 + \alpha_1 + \alpha_2 = \alpha \\ |\alpha_0|,|\alpha_2| < |\alpha|}} \partial^{\alpha_0} W^{\dagger} \partial^{\alpha_1} W \partial^{\alpha_2} W^{\dagger}.
	\end{align}
	From $W^{\dagger} = W^{\dagger} Q$ and $W^{\dagger} = P W^{\dagger}$, we get
	\begin{align}
		\partial^{\alpha} W^{\dagger}
		&= \sum_{\alpha_0 + \alpha_1 = \alpha} \partial^{\alpha_0} W^{\dagger} \partial^{\alpha_1} Q
		= \partial^{\alpha} W^{\dagger} Q + \sum_{\substack{\alpha_0 + \alpha_1 = \alpha \\ |\alpha_0| < |\alpha|}} \partial^{\alpha_0} W^{\dagger} \partial^{\alpha_1} Q, \label{WQ} \\
		\partial^{\alpha} W^{\dagger}
		&= \sum_{\alpha_0 + \alpha_1 = \alpha} \partial^{\alpha_0} P \partial^{\alpha_1} W^{\dagger}
		= P \partial^{\alpha} W^{\dagger} + \sum_{\substack{\alpha_0 + \alpha_1 = \alpha \\ |\alpha_1| < |\alpha|}} \partial^{\alpha_0} P \partial^{\alpha_1} W^{\dagger}. \label{PW}
	\end{align}
	Replacing the terms $\partial^{\alpha} W^{\dagger} Q$ and $P \partial^{\alpha} W^{\dagger}$ in \eqref{main} by their expression in \eqref{WQ}-\eqref{PW}, we obtain
	\begin{align*}
		\partial^{\alpha} W^{\dagger}
		& = \sum_{\substack{\alpha_0 + \alpha_1 = \alpha \\ |\alpha_0| < |\alpha|}} \partial^{\alpha_0} W^{\dagger} \partial^{\alpha_1} Q \\
		& \quad + \sum_{\substack{\alpha_0 + \alpha_1 = \alpha \\ |\alpha_1| < |\alpha|}} \partial^{\alpha_0} P \partial^{\alpha_1} W^{\dagger} - \sum_{\substack{\alpha_0 + \alpha_1 + \alpha_2 = \alpha \\ |\alpha_0|,|\alpha_2| < |\alpha|}} \partial^{\alpha_0} W^{\dagger} \partial^{\alpha_1} W \partial^{\alpha_2} W^{\dagger}.
	\end{align*}
	In particular, for $|\alpha| = 1$, this reduces to
	\begin{align*}
		\partial^{\alpha} W^{\dagger}
		= W^{\dagger} \partial^{\alpha} Q + \partial^{\alpha} P W^{\dagger} - W^{\dagger} \partial^{\alpha} W W^{\dagger}.
	\end{align*}
	Since \eqref{ind:QP} holds for $k = 1$, we deduce that $\|\partial^{\alpha} W^{\dagger}\| \leqslant C_1 \|W^{\dagger}\|^{2}$, so that \eqref{ind:W} holds for $k = 1$. Assuming now that \eqref{ind:QP} holds for $k \leqslant K+1$, and \eqref{ind:W} holds for $k \leqslant K$, let $\alpha$ be a multi-index with $|\alpha| = K+1$. Then, for constants $C$ depending on $(C_{j})_{j \leqslant K}$, one has
	\begin{align}\label{almostfinal}
		\|\partial^{\alpha} W^{\dagger}\|
		&\leqslant \sum_{\substack{\alpha_0 + \alpha_1 = \alpha \\ |\alpha_0| < |\alpha|}} C \|W^{\dagger}\|^{|\alpha_0|+1} \|W^{\dagger}\|^{|\alpha_1|} + \sum_{\substack{\alpha_0 + \alpha_1 = \alpha \\ |\alpha_1| < |\alpha|}} C \|W^{\dagger}\|^{|\alpha_0|} \|W^{\dagger}\|^{|\alpha_1|+1} \notag \\
		&\quad+ \sum_{\substack{\alpha_0 + \alpha_1 + \alpha_2 = \alpha \\ |\alpha_0|,|\alpha_2| < |\alpha| \\ |\alpha_1| > 0}} C \|W^{\dagger}\|^{|\alpha_0|+1} \|W^{\dagger}\|^{|\alpha_2|+1} +
		\sum_{\substack{\alpha_0 + \alpha_2 = \alpha \\ |\alpha_0|,|\alpha_2| < |\alpha|}} \|\partial^{\alpha_0} W^{\dagger} W \partial^{\alpha_2} W^{\dagger}\| \notag \\
		&= C \|W^{\dagger}\|^{|\alpha|+1} + \sum_{\substack{\alpha_0 + \alpha_2 = \alpha \\ |\alpha_0|,|\alpha_2| < |\alpha|}} C \|W^{\dagger}\|^{|\alpha_0|+1} \|W \partial^{\alpha_2} W^{\dagger}\|.
	\end{align}
	If we were to bound the term $\|W \partial^{\alpha_2} W^{\dagger}\|$ by $C_{|\alpha_2|} \|W^{\dagger}\|^{|\alpha_2|+1}$ using the induction assumption, then we would obtain an exponent of $|\alpha_0|+1 + |\alpha_2|+1 = |\alpha|+2$ on the term $\|W^{\dagger}\|$ in the last sum, instead of $|\alpha|+1$. However, the term $W$ allows for a better bound: indeed, taking derivatives of the identity $Q = W W^{\dagger}$, we get
	\begin{align*}
		\partial^{\alpha} Q
		= \sum_{\alpha_0 + \alpha_1 = \alpha} \partial^{\alpha_0} W \partial^{\alpha_1} W^{\dagger}
		= \sum_{\substack{\alpha_0 + \alpha_1 = \alpha \\ |\alpha_1| < |\alpha|}} \partial^{\alpha_0} W \partial^{\alpha_1} W^{\dagger} + W \partial^{\alpha} W^{\dagger}.
	\end{align*}
	Isolating the term $W \partial^{\alpha} W^{\dagger}$ and using \eqref{ind:QP}-\eqref{ind:W}, we get
	\begin{align*}
		\|W \partial^{\alpha} W^{\dagger}\|
		\leqslant C_{|\alpha|} \|W^{\dagger}\|^{|\alpha|} + \sum_{\substack{\alpha_0 + \alpha_1 = \alpha \\ |\alpha_1| < |\alpha|}}C_{|\alpha_1|} \|W^{\dagger}\|^{|\alpha_1|+1}
		\leqslant C_{K+1} \|W^{\dagger}\|^{|\alpha|}.
	\end{align*}
	Plugging this in \eqref{almostfinal}, we conclude to the bound $\|\partial^{\alpha} W^{\dagger}\| \leqslant C_{K+1} \|W^{\dagger}\|^{|\alpha|+1}$.
\end{proof}
%%%%%%%%%%%%%%%%%%%%%%%%%%%%%%%%%%%%%%%%%%%%%%%%%%%%%%%%%%%%%%%%%%%%%%%%%%%%%%
%%  B.  TRANSFER OF MULTIPLIERS TO THE TORUS
%%%%%%%%%%%%%%%%%%%%%%%%%%%%%%%%%%%%%%%%%%%%%%%%%%%%%%%%%%%%%%%%%%%%%%%%%%%%%%
\section{Transfer of multipliers to the torus}\label{appx:transfer}
\begin{proof}[Proof of \Cref{res:transfer}]
	By \cite[Lemma~2]{raitaPotentialsAquasiconvexity2019}, there exists a linear homogeneous differential operator $\Bsp : \regC^{\infty}(\R^d;\presou) \to \regC^{\infty}(\R^d;\sou)$ such that $\mathrm{Im} \,\Bsp \subset \ker \Asp$, and any $w$ as in the statement writes as $w = \Bsp v$ for some periodic $v \in \regC^{\infty}(\mathbb{T}^d;\presou)$. Write
	\[
	\Bsp = \sum_{|\gamma| = \ell} B_{\gamma}\, \partial^{\gamma}, \qquad \ell \coloneqq \deg \Bsp,
	\]
	and let $v_{\square} \in \regC^{\infty}(\R^d;\presou)$ be the periodic extension of $v$, so that $w_{\square} \coloneqq \Bsp v_{\square}$ is the periodic extension of $w$. We may assume $\ell \ge 1$; if $\ell = 0$ the operator $\Bsp$ is algebraic, the commutator \eqref{transfer:commutator} below vanishes identically, and the argument simplifies.

	\emph{Step 1: the cut-off and its commutator.} Fix once and for all $\chi \in \regC^{\infty}(\R;[0,1])$ with $\chi \equiv 1$ on $(-\infty,0]$ and $\chi \equiv 0$ on $[1,\infty)$, and put, for $N \in \mathbb N_*$,
	\[
	\varphi_N(x) \coloneqq \prod_{i=1}^{d} \chi(|x_i| - N) .
	\]
	Each $\varphi_N$ is smooth --- near $\{x_i = 0\}$ the $i$-th factor is constantly $1$, because $N \ge 1$ --- takes values in $[0,1]$, equals $1$ on $[-N,N]^d$ and vanishes outside $[-(N+1),N+1]^d$. Being assembled from translates and reflections of the single profile $\chi$, it also satisfies
	\begin{align}\label{transfer:phibound}
		\begin{aligned}
			C_{\varphi} &\coloneqq \sup_{N \in \mathbb N_*} \ \max_{1 \le |\beta| \le \ell} \ \|\partial^{\beta}\varphi_N\|_{\infty} \ <\ \infty , \\
			\mathrm{spt}\,\partial^{\beta}\varphi_N &\subset A_N \coloneqq [-(N+1),N+1]^d \mysm (-N,N)^d
		\end{aligned}
	\end{align}
	for every multi-index $\beta \ne 0$.

	Set $w_N \coloneqq \Bsp(\varphi_N v_{\square})$. It is smooth, compactly supported, and $\Asp$-free because $\mathrm{Im}\,\Bsp \subset \ker \Asp$; the hypothesis of the lemma therefore applies to it:
	\begin{align}\label{transfer:hyp}
		\|w_N\|_{L^{1,\infty}} \leqslant \const{V} \|\dist(w_N,V)\|_{L^1}.
	\end{align}
	By the Leibniz rule, $w_N$ differs from $\varphi_N w_{\square}$ exactly by a commutator,
	\begin{align}\label{transfer:commutator}
		w_N - \varphi_N w_{\square}
		\ =\ [\Bsp,\varphi_N]\, v_{\square}
		\ =\ \sum_{|\gamma| = \ell} \ \sum_{0 \neq \beta \leqslant \gamma} \binom{\gamma}{\beta} \, B_{\gamma} \ \partial^{\beta}\varphi_N \ \partial^{\gamma - \beta} v_{\square} ,
	\end{align}
	which is a differential operator of order at most $\ell - 1$ applied to $v_{\square}$, with coefficients supported in $A_N$ by \eqref{transfer:phibound}. Consequently
	\[
	\|w_N - \varphi_N w_{\square}\|_{L^1}
	\ \leqslant\ C(\Bsp)\, C_{\varphi} \sum_{|\alpha| \leqslant \ell - 1} \|\partial^{\alpha} v_{\square}\|_{L^1(A_N)} .
	\]
	The annulus $A_N$ is covered by at most $C_d\, N^{d-1}$ unit cubes with integer vertices, and $v_{\square}$ is $\mathbb Z^d$-periodic, so that $\|\partial^{\alpha} v_{\square}\|_{L^1(A_N)} \leqslant C_d\, N^{d-1} \|\partial^{\alpha} v\|_{L^1(\mathbb T^d)}$ for every $\alpha$. Altogether,
	\begin{align}\label{transfer:commbound}
		\|w_N - \varphi_N w_{\square}\|_{L^1} \ \leqslant\ C(v,\Bsp,d)\, N^{d-1} .
	\end{align}

	\emph{Step 2: lower bound.} On the open cube $(-N,N)^d$ the function $\varphi_N$ is identically $1$, hence $w_N = \Bsp v_{\square} = w_{\square}$ there. Since $w_{\square}$ is $\mathbb Z^d$-periodic and $(-N,N)^d$ is the union of $(2N)^d$ unit cells, for every $\lambda > 0$
	\[
	\mathscr L^d(\{|w_N| > \lambda\}) \ \geqslant\ \mathscr L^d\big(\{|w_{\square}| > \lambda\} \cap (-N,N)^d\big)
	\ =\ (2N)^d \, \mathscr L^d_{\mathbb T^d}(\{|w| > \lambda\}) ,
	\]
	and therefore
	\begin{align}\label{transfer:lower}
		\|w_N\|_{L^{1,\infty}} \ \geqslant\ (2N)^d \, \|w\|_{L^{1,\infty}(\mathbb T^d)} .
	\end{align}

	\emph{Step 3: upper bound.} As $\dist(\,\cdot\,,V)$ is $1$-Lipschitz, \eqref{transfer:commutator} gives the pointwise bound $\dist(w_N,V) \leqslant \dist(\varphi_N w_{\square},V) + |w_N - \varphi_N w_{\square}|$. Since $\dist(\,\cdot\,,V)$ is positively $1$-homogeneous and $\varphi_N$ takes values in $[0,1]$, one has $\dist(\varphi_N(x)\var, V) = \varphi_N(x)\dist(\var,V) \leqslant \dist(\var,V)$ for every $\var \in \sou$; as $\{\varphi_N > 0\}$ is covered by $(2(N+1))^d$ unit cells, integrating and inserting \eqref{transfer:commbound} yields
	\begin{align}\label{transfer:upper}
		\|\dist(w_N,V)\|_{L^1}
		\ \leqslant\ (2(N+1))^d \|\dist(w,V)\|_{L^1(\mathbb T^d)} + C(v,\Bsp,d)\, N^{d-1} .
	\end{align}

	\emph{Step 4: conclusion.} Inserting \eqref{transfer:lower} and \eqref{transfer:upper} into \eqref{transfer:hyp} and dividing by $(2N)^d$,
	\[
	\|w\|_{L^{1,\infty}(\mathbb T^d)}
	\ \leqslant\ \const{V} \left[ \Big(\frac{N+1}{N}\Big)^{d} \|\dist(w,V)\|_{L^1(\mathbb T^d)} \ +\ \frac{C(v,\Bsp,d)}{2^{d}\, N} \right] .
	\]
	Letting $N \to \infty$, so that $((N+1)/N)^d \to 1$ and the last term vanishes, the conclusion follows.
\end{proof}

%%%%%%%%%%%%%%%%%%%%%%%%%%%%%%%%%%%%%%%%%%%%%%%%%%%%%%%%%%%%%%%%%%%%%%%%%%%%%%
%%  ACKNOWLEDGEMENTS
%%%%%%%%%%%%%%%%%%%%%%%%%%%%%%%%%%%%%%%%%%%%%%%%%%%%%%%%%%%%%%%%%%%%%%%%%%%%%%
\section*{Acknowledgements}
Both authors are supported by the European Research Council (ERC) Starting Grant ``ConFine'' (Grant No.\ 101078057) at the Universit\`a di Pisa. The first author wishes to thank Filip Rindler for sharing the question that motivated this work and for several fruitful discussions along the years.

%%%%%%%%%%%%%%%%%%%%%%%%%%%%%%%%%%%%%%%%%%%%%%%%%%%%%%%%%%%%%%%%%%%%%%%%%%%%%%
%%  DATA AVAILABILITY AND CONFLICT OF INTEREST
%%%%%%%%%%%%%%%%%%%%%%%%%%%%%%%%%%%%%%%%%%%%%%%%%%%%%%%%%%%%%%%%%%%%%%%%%%%%%%
%\section*{Data availability and conflict of interest}
%The datasets and source files supporting the conclusions of this article are available in the arXiv repository (\url{https://arxiv.org/abs/2606.16762}) in \texttt{.tex} and \texttt{.pdf} formats. The authors declare that they have no conflict of interest.

%%%%%%%%%%%%%%%%%%%%%%%%%%%%%%%%%%%%%%%%%%%%%%%%%%%%%%%%%%%%%%%%%%%%%%%%%%%%%%
%%  BIBLIOGRAPHY
%%%%%%%%%%%%%%%%%%%%%%%%%%%%%%%%%%%%%%%%%%%%%%%%%%%%%%%%%%%%%%%%%%%%%%%%%%%%%%
\printbibliography[heading=bibintoc]
\end{document}